\documentclass[11pt]{amsart}
\usepackage{amssymb,latexsym,amsmath,graphicx,graphics,epic,eepic}

\theoremstyle{plain}
\newtheorem{theorem}{Theorem}
\numberwithin{theorem}{section}
\newtheorem{lemma}[theorem]{Lemma}
\newtheorem{proposition}[theorem]{Proposition}
\newtheorem{corollary}[theorem]{Corollary}

\theoremstyle{definition}

\newtheorem{example}[theorem]{Example}

\newtheorem{question}[theorem]{Question}
\newtheorem{remark}[theorem]{Remark}

\newcommand{\C}{{\mathbb C}}
\newcommand{\R}{{\mathbb R}}
\newcommand{\Z}{{\mathbb Z}}
\newcommand{\Q}{{\mathbb Q}}
\renewcommand{\P}{{\mathbb P}}
\newcommand{\s}{{\mathbb S}}

\begin{document}
\title{On the orders of periodic diffeomorphisms of $4$-manifolds}
\author{Weimin Chen}
\thanks{The author was supported in part by NSF grant DMS-0603932}
\date{}
\maketitle

\begin{abstract}
This paper initiated an investigation on the following question: Suppose a smooth $4$-manifold 
does not admit any smooth circle actions. Does there exist a constant $C>0$ such that the manifold
support no smooth $\Z_p$-actions of prime order for $p>C$? We gave affirmative results to this
question for the case of holomorphic and symplectic actions, with an interesting finding that the 
constant $C$ in the holomorphic case is topological in nature while in the symplectic case it involves
also the smooth structure of the manifold. 
\end{abstract}

\section{Introduction}

A classical theorem of Hurwitz says that, for a complex curve $\Sigma$ of genus
$g\geq 2$, the order of its automorphism group $\text{Aut}(\Sigma)$ satisfies the following 
topological bound:
$$
|\text{Aut}(\Sigma)|\leq 84(g-1)=42\; \text{deg }K_\Sigma.
$$
Various attempts have been made to generalize this result to higher-dimensional projective varieties.
For a minimal smooth projective surface $X$ of general type, Xiao obtained the optimal result in 
\cite{Xiao1, Xiao2} that $|\text{Aut}(X)|\leq 42^2 \; c_1(K_X)^2$, after a series of earlier work by 
Andreotti \cite{An}, Howard and Sommese \cite{HoSo}, Corti \cite{Corti}, Huckleberry and Sauer 
\cite{HuSa} and Xiao \cite{Xiao0}. For dimensions greater than $2$, see recent work 
of D.-Q. Zhang \cite{Zhang2} and the references therein. 

The purpose of this paper is to investigate generalizations of Hurwitz's theorem to finite automorphism
groups of smooth $4$-manifolds.  Our central theme may be summarized in the following

\vspace{2mm}

{\it Main Question: Suppose $X$ is a smooth $4$-manifold which does not admit any smooth circle
actions.  Does there exist a constant $C>0$ such that there are no nontrivial smooth
$\Z_p$-actions of prime order on $X$ provided that $p>C$? Moreover, assuming there is such a 
constant, what structures (e.g., homology, homotopy, smooth structures, etc.) of $X$ does it
depend on?
}
\vspace{2mm}

We remark that for locally linear topological actions the analogous question is already known to have 
a negative answer. Indeed, on the one hand, Edmonds in \cite{Ed} showed that for any simply connected 
$4$-manifold $X$ and prime number $p>3$, there exists a homologically trivial, locally linear, pseudofree topological $\Z_p$-action on $X$, and on the other hand, by work of Fintushel \cite{F, F1},  a simply connected $4$-manifold with non-zero signature and even intersection form does not admit any locally 
linear topological circle actions. 

In this paper, we establish the existence of such a constant $C$ for the more restricted classes of actions,
i.e., holomorphic and symplectic $\Z_p$-actions, hoping to shed some light to the problem in general. 
An interesting phenomenon we discovered is that the constants for these two classes of actions are of different nature; in the holomorphic case the constant $C$ depends only on the integral homology of the 
manifold, while in the symplectic case, it depends also on the smooth structure. 

Before we state the theorems, it is helpful to recall the relevant results regarding the existence of
smooth circle actions on $4$-manifolds. First, with the resolution of the $3$-dimensional Poincar\'{e}
conjecture by Perelman, a theorem of Fintushel in \cite{F1} may be strengthened to the following: 
a simply connected $4$-manifold admitting a smooth circle action must be a connected 
sum of copies of $\s^4$, $\pm \C\P^2$, and $\s^2\times\s^2$. 
For the non-simply connected case, one has the following useful criterion due to Baldridge \cite{Bald} : 
Let $X$ be a $4$-manifold admitting a smooth circle action with nonempty fixed point set. 
Then $X$ has vanishing Seiberg-Witten invariant when $b_2^{+}>1$, and when $b_2^{+}=1$ and $X$ is symplectic, $X$ is diffeomorphic to a rational or ruled surface. For fixed-point free smooth circle
actions, formulas relating the Seiberg-Witten invariants of the $4$-manifold and the quotient $3$-orbifold
were given in Baldridge \cite{B1,B2}. Note that if a $4$-manifold admits a fixed-point free circle
action, the Euler characteristic and the signature of the manifold must vanish. 

With the preceding understood, we first state the result for the case of holomorphic actions 
on complex surfaces --- these are the primary examples of smooth actions on $4$-manifolds. 
The proof of this theorem is based on the aforementioned Xiao's generalization \cite{Xiao1,Xiao2} of Hurwitz's theorem and the homological rigidity of holomorphic actions on elliptic surfaces as in
Ueno \cite{Ueno} and Peters \cite{Peters}. 

\begin{theorem}
Let $X$ be a compact complex surface with Kodaira dimension $\kappa(X)\geq 0$. Suppose 
$X$ does not admit any holomorphic circle actions. Then there exists a constant $C>0$ such that
there are no nontrivial holomorphic $\Z_p$-actions of prime order on $X$ provided that $p>C$. 
Moreover, the constant $C$ depends linearly on the Betti numbers of $X$ and the order of the
torsion subgroup of $H_2(X)$, i.e., there exists a universal constant $c>0$ such that
$$
C=c(1+b_1+b_2+|Tor\; H_2|).
$$
\end{theorem}

\begin{remark}
(1) According to the Enriques-Kodaira classification of complex surfaces (cf. \cite{BPV}), a complex
surface of negative Kodaira dimension falls into two disjoint groups, the rational or ruled surfaces and the
surfaces of type VII (which are non-K\"{a}hler). The exclusion of the former is natural as each rational or ruled surface admits 
a smooth circle action, while the exclusion of the latter is due to lack of understanding. 

(2) The assumption of no holomorphic circle actions is potentially weaker than that of no smooth circle actions. On the other hand, for a non-K\"{a}hler complex surface, there is no easy criterion which 
excludes the existence of smooth circle actions. 

(3) The constant $C$ in Theorem 1.1 is sharp in the following sense. There are complex surfaces with
holomorphic $\Z_p$-actions whose Betti numbers or the order of the torsion subgroup of $H_2$ grow
linearly with $p$. For example, let $\Sigma$ be a complex curve of genus $g=p+1$ which admits a free
$\Z_p$-action, and let $\Sigma^\prime$ be a complex curve of genus $2$. Then the complex surface 
$\Sigma\times \Sigma^\prime$, which is a surface of general type and 
does not support any holomorphic circle actions, admits a 
holomorphic $\Z_p$-action. The Betti numbers of  $\Sigma\times \Sigma^\prime$ grow linearly with $p$. 

For an example which shows the necessity of inclusion of $|Tor \; H_2|$ in the constant $C$, we consider
the rational elliptic surface $X_0$ defined by the Weierstrass equation
$$
y^2z=x^3+v^5z^3.
$$
For any prime number $p\geq 5$, one can define an order-$p$ automorphism 
$g$ of $X_0$ as follows (cf. \cite{Zhang1}):
$$
g: (x,y,z;v)\mapsto (\mu_p^{-5} x,y,\mu_p^{-15}z;\mu_p^6v),\;\; 
\mu_p\equiv\exp(2\pi i/p).
$$
Note that $g$ preserves the elliptic fibration and leaves exactly the two 
singular fibers (at $v=0$ and $v=\infty$) invariant. For any $m>1$, let $X_{m,p}$ be the 
elliptic surface obtained from $g$-equivariantly performing log transforms to $X_0$ with 
multiplicity $m$ on a $g$-invariant set of regular fibers. 
Then $X_{m,p}$ has Kodaira dimension $1$ and does
not admit any holomorphic circle actions. Moreover, $X_{m,p}$ naturally inherits a holomorphic 
$\Z_p$-action from $X_0$. With this understood, note that 
$$
Tor \; H_2(X_{m,p})=(\Z_m)^{p-1}. 
$$
\end{remark}

In the above construction, if we replace the log transforms by the Fintushel-Stern knot surgery
\cite{FS}, we arrive at examples which show that in the symplectic case, Betti numbers and the 
order of torsion subgroup of $H_2$ alone are no longer sufficient. 
The point is that knot surgery does not change the integral homology of the manifold; in particular, 
it does not create torsion subgroups. Rather, it actually preserves the simply connectedness of the 
manifold.  We spell out the details in the following 

\begin{example} (Knot surgery on parallel copies of a regular fiber of $X_0$)

Let $X$ be a simply connected elliptic surface with a section (e.g. $X_0$), and let $T$ be a regular 
fiber in $X$. Then $T$ is c-embedded in the sense of Fintushel-Stern \cite{FS}, i.e., $T$ is contained 
in a neighborhood of a cusp fiber. To see this, by a theorem of Moishezon (cf. Friedmann-Morgan
\cite{FM}, Theorem 4.8 in \S 1.4 and Theorem 3.6 in \S 2.3), $X$ may be slightly perturbed through 
elliptic surfaces to an elliptic surface $X^\prime$ of only nodal singular fibers (i.e., of type $I_1$), 
and moreover, the monodromy representation of
$X^\prime$ can be put into a standard form. From this one can easily create a cusp fiber in $X$ by
combining an appropriately chosen pair of nodal singular fibers in $X^\prime$. This proves that 
$T$ is c-embedded. Moreover, $\pi_1(X\setminus T)$ is trivial because $X$ has a section so that
the meridian of $T$ bounds a disc in $X\setminus T$. 

These two conditions, i.e., $T$ is c-embedded and $\pi_1(X\setminus T)$ is trivial, are the assumptions
made in Fintushel-Stern \cite{FS} for the knot surgery, in particular, to ensure that the resulting 
$4$-manifold continues to be simply connected. The discussion in the previous paragraph shows 
that one
can perform knot surgery to $X$ along a regular fiber. However, if one performs repeated knot surgery
along parallel copies of a regular fiber, it is not clear that the resulting $4$-manifold is still simply 
connected (because the triviality of $\pi_1(X\setminus T)$ might not be preserved). 
But we should point out that regardless of the assumptions, it is easily seen from a
Mayer-Vietoris argument that knot surgery always preserves the integral homology of the manifold. 

The following observation was kindly pointed out to the author by an anonymous referee: {\it Performing 
repeated knot surgery along parallel copies is equivalent to performing a single knot surgery using the
connected sum of the knots}.  This observation can be easily seen from the following identification:
$$
Y_K \cup_\phi (\s^1\times [0,1]\times [0,1] )\cup_\psi Y_{K^\prime}= Y_{K\# K^\prime}, 
$$
where $Y_K$ stands for $\s^3$ with a regular neighborhood of a knot $K$ removed, and the maps 
$\phi$ identifies $\s^1\times [0,1]\times \{0\}$ to part of $\partial Y_K$ and $\psi$ identifies 
$\s^1\times [0,1]\times \{1\}$ to part of $\partial Y_{K^\prime}$, sending $\s^1$ to the meridian of the
knots. In conclusion, one can perform repeated knot surgery to $X$ along parallel copies of a regular fiber.
In particular, the resulting $4$-manifold continues to be homeomorphic to $X$. 

With the preceding understood, let $X_p$ be the symplectic $4$-manifold obtained 
by performing repeated knot surgery to the rational elliptic surface $X_0$ using the trefoil knot 
on $p$ copies of regular fibers of the elliptic fibration which are invariant under the order-$p$ 
automorphism $g$:
$$
g: (x,y,z;v)\mapsto (\mu_p^{-5} x,y,\mu_p^{-15}z;\mu_p^6v),\;\; 
\mu_p\equiv\exp(2\pi i/p).
$$
We require that the knot surgeries are done equivariantly with respect to $g$,
so that there is an induced symplectomorphism of order $p$ on $X_p$. By our previous discussion,
$X_p$ is homeomorphic to $X_0$.

To see that $X_p$ does not admit any smooth circle actions, we note that 
the canonical class of $X_p$ is given by the formula 
$$
c_1(K_{X_p})=(2p-1)\cdot [T], 
$$
where $[T]$ is the fiber class of the elliptic fibration (cf. \cite{FS}). 
Let $\omega$ be the symplectic structure on $X_p$, which pairs positively with $[T]$. Then
$(X_p,\omega)$ satisfies $c_1(K_{X_p})\cdot [\omega]>0$ and $c_1(K_{X_p})^2=0$. This 
implies that $X_p$ is neither rational nor ruled, cf. \cite{Li}. By Baldridge's theorem \cite{Bald}
or the strengthened version of the theorem of Fintushel in \cite{F1}, $X_p$ does not admit any 
smooth circle actions. 
 
We thus obtained, for any prime number $p\geq 5$, a symplectic 
$4$-manifold $X_p$ homeomorphic to the rational elliptic surface, 
which admits no smooth circle actions but has 
a periodic symplectomorphism of order $p$. This shows that Betti numbers and $|Tor\; H_2|$ 
alone are not sufficient 
in the case of symplectic $\Z_p$-actions. 
\end{example}

In the above example, note that the canonical class $c_1(K_{X_p})$ is a multiple class whose 
multiplicity, $2p-1$, grows linearly with $p$. To better exploit this fact, let's assume the symplectic 
structure $\omega$ on $X_p$ is integral, i.e., $[\omega]\in H^2(X_p;\Z)$. (This can be arranged.)
In particular, $[\omega]\cdot [T]\geq 1$. This gives rise to the following bound
$$
p\leq \frac{1}{2}(c_1(K_{X_p})\cdot [\omega]+1).
$$

Now we state the Main Theorem which established the constant $C$ for symplectic $\Z_p$-actions.
Let $(X,\omega)$ be a symplectic $4$-manifold such that $[\omega]\in H^2(X;\Q)$.  Denote by 
$N_\omega$ the smallest positive integer such that $[N_\omega\omega]$ is an integral class. We set
$$
C_\omega\equiv N_\omega c_1(K)\cdot [\omega],
$$ 
where $K$ is the canonical bundle of $(X,\omega)$. 

\begin{theorem}
{\em(}Main Theorem{\em)}  Let $(X,\omega)$ be a symplectic $4$-manifold with $b^{+}_2>1$ and
$[\omega]\in H^2(X;\Q)$, which does not satisfy the following condition: $X$ is minimal with vanishing
Euler characteristic and signature. Then there exists a constant $C>0$:
$$
C=c(1+b_1^2+b_2^2)C_\omega^2
$$
where $c>0$ is a universal constant,  such that there are no nontrivial symplectic $\Z_p$-actions 
of prime order on $(X,\omega)$ provided that $p>C$. 
\end{theorem}

A few remarks are in order. First, since $c_1(K)$ is a Seiberg-Witten basic class of $X$, the constant 
$C$ in Theorem 1.4 depends on the smooth structure of $X$ through the inclusion of $C_\omega$. 
Secondly, the assumption that $X$ is not of vanishing Euler characteristic and signature is potentially stronger than the assumption that $X$ admits no smooth circle actions. By Baldridge's theorem \cite{Bald},
$X$ admits at most fixed-point free smooth circle actions. However,  there is no easy criterion excluding 
this other than the assumption that $X$ is not of vanishing Euler characteristic and signature. 

\begin{question}
(1) How essential is the assumption $[\omega]\in H^2(X;\Q)$? (Notice that the constant $C_\omega$ depends on $[\omega]$ in a rather unstable fashion because of the factor $N_\omega$, therefore one 
can not remove the assumption that $[\omega]$ is rational by simply perturbing $\omega$ into one 
which is of rational class.)

(2) (Symplectic Hurwitz-Xiao, cf. \cite{Xiao1,Xiao2}) 
Let $(X,\omega)$ be a minimal symplectic $4$-manifold of Kodaira
dimension $2$ (i.e., $c_1(K)\cdot [\omega]>0$ and $c_1(K)^2>0$, cf. \cite{Li}). Does there exists a 
universal constant $c>0$, such that
$$
|G|\leq c \cdot c_1(K)^2
$$
for any finite group $G$ of symplectomorphisms of $(X,\omega)$? (Note that the manifolds $X_p$ 
in Example 1.3 are minimal (cf. Gompf \cite{G} and Usher \cite{Usher}) with $c_1(K_{X_p})^2=0$,
so they are not of Kodaira dimension $2$ and do not give counterexamples.) 

(3) For the general case of smooth $\Z_p$-actions, is the constant $C$ in the Main Question 
``describable" in terms of invariants of the manifold? Will the homotopy groups of the manifold play a role in the constant? What invariants of the smooth structure (e.g.  SW basic classes, SW invariants, etc.) will enter the constant and how do they enter the constant? (Note that Example 1.3 indicates the necessity of including information about the smooth structure in the constant. In the case of symplectic 
$\Z_p$-actions, the pairing $c_1(K)\cdot [\omega]$ is used.)

(4) Let $X$ be a simply connected smoothable $4$-manifold with even intersection form and non-zero signature. Does there exist a constant $C>0$ depending only on the topological type of $X$, such 
that for any prime number $p>C$, there are no $\Z_p$-actions on $X$ which are smooth with respect to
some smooth structure? (Note that the existence of smooth circle actions depends on the smooth 
structure of the manifold, as shown by Example 1.3. By Atiyah and Hirzebruch \cite{AH},
a spin $4$-manifold with non-zero signature does not admit smooth circle actions for any given smooth
structure.)
\end{question}

We remark that Question 1.5(4) is particularly interesting in the case of $K3$ 
surfaces. It is a long-standing problem as whether the $K3$ surface 
(with the standard smooth structure) admits any smooth finite group actions which are 
homologically trivial. It is well-known that there are no such holomorphic actions (cf. \cite{BPV}), and 
recently it was shown that there are no such symplectic 
finite group actions as well, cf. \cite{CK}. Since for $p>23$, any 
$\Z_p$-action of prime order on a homotopy $K3$ surface is homologically trivial, we 
see that Question 1.5(4) is related to the above homological rigidity problem of smooth actions 
on the $K3$ surface (or more generally, on a homotopy $K3$ surface). 

\vspace{3mm}

Next we discuss the main ideas and ingredients in the proof of the Main Theorem. 
The proof begins with a two-step reduction. First, based on an algebraic result concerning 
integral $\Z_p$-representations we showed that for sufficiently large $p$, say $p>2+b_1+b_2$,
any continuous $\Z_p$-action on $X$ must be homologically trivial over $\Q$
(see Lemma 2.1 for details). Secondly, we showed that one can always equivariantly blow down 
$X$ if it is not minimal. (This is trivial in the holomorphic case; for symplectic actions the proof 
is more involved and the claim is not true in the rational or ruled case, see Lemmas 2.2, 2.3.) 
After the two-step reduction, we may assume for simplicity that $X$ is minimal and the $\Z_p$-action
is homologically trivial over $\Q$.

The main technical ingredient is Taubes' work \cite{T} on the Seiberg-Witten and Gromov invariants 
of symplectic $4$-manifolds.  Let $(M,\omega)$ be
a symplectic $4$-manifold and $G$ be a finite group acting on $M$ smoothly and effectively which
preserves the symplectic form $\omega$. Denote by $b_G^{2,+}$ the dimension of the maximal 
subspace of $H^2(M;\R)$ over which the cup-product is positive and the induced action of $G$ 
is trivial. Then an equivariant version of Taubes' theorem $SW\Rightarrow Gr$ (cf. \cite{T}) 
may be applied to $(M,\omega)$ provided that $b_G^{2,+}\geq 2$. More precisely,  when 
$b_G^{2,+}\geq 2$, 
the $G$-equivariant Seiberg-Witten invariant is well-defined and is non-zero for the $G$-equivariant
canonical bundle $K_\omega$. This implies that, for any $r>0$, the $r$-version of Taubes' perturbed 
Seiberg-Witten equations has a solution $((\alpha,\beta),a)$ which is fixed under the action of $G$. 
Letting $r\rightarrow \infty$ as usual, the zero set $\alpha^{-1}(0)$ converges to a finite set of $J$-holomorphic curves $\{C_i\}$, such that $c_1(K_\omega)=\sum_i n_i C_i$ for some integers $n_i>0$. 
Here $J$ is any fixed choice of $G$-equivariant, $\omega$-compatible almost complex structure. 

Since $\alpha$ is fixed under $G$, it follows easily that the set $\cup_i C_i$ is $G$-invariant, and 
furthermore, $\cup_i C_i$ contains all the fixed points of $G$ except for those isolated ones at which
the representation of $G$ on the complex tangent space has determinant $1$. This in principle allows
one to analyze the action of $G$ near its fixed point set,  and sometimes even the induced representation
of $G$ on the second cohomology --- the two crucial pieces of informations about the action of $G$ 
on $M$ --- by looking at the restriction of the $G$-action in a neighborhood of $\cup_i C_i$. The main difficulty lies in the fact that in general not much can be said about the structure of the set 
$\cup_i C_i$. Unlike 
the non-equivariant case where $\{C_i\}$ can be made disjoint and embedded for a generic choice
of $J$, in the presence of group actions the set $\cup_i C_i$ could be very complicated in general
even with a choice of generic equivariant $J$, cf. \cite{C1}. The only exceptional case is when
$(M,\omega)$ is minimal and $c_1(K_\omega)^2=0$. This was explored in \cite{CK} in investigating 
the homological rigidity of symplectic finite group actions. For further applications concerning group
actions and exotic smooth structures, see \cite{CK1, CK2}. 

With the preceding understood, the proof of the Main Theorem relies in a crucial way on the following
technical lemma. Recall from \cite{T}, Section 5(e), that for any point $x\in \cup_i C_i$, and for any 
embedded $J$-holomorphic disk $D$ such that $D\cap (\cup_i C_i)=\{x\}$, a local intersection number
$\text{int}_D(x)$ is defined. (Note that it was shown in \cite{T} that such embedded $J$-holomorphic 
disks $D$ exist in abundance.) 

\begin{lemma}
Let $1\neq g\in G$ and $x\in \cup_i C_i$ such that $g\cdot x=x$. Suppose the action of $g$ near $x$ 
is given by
$$
g\cdot (z_1,z_2)=(\lambda^{m_1}z_1,\lambda^{m_2}z_2)
$$ 
in an $\omega$-compatible local complex coordinate system $(z_1,z_2)$ centered at $x$,
where $\lambda=\exp(2\pi i/m)$ with $m\equiv \text{order}(g)$, and $0\leq m_1,m_2<m$. 
Suppose further that $x\in\alpha^{-1}(0)$ for all $r>0$. Then there exist non-negative integers 
$a_1,a_2$ with $a_1+a_2 >0$ satisfying the congruence relation
$$
(a_1+1)m_1+(a_2+1)m_2 = 0 \pmod{m},
$$
such that 
$$
\text{int}_D(x)\geq a_1+a_2
$$ 
for any embedded $J$-holomorphic disk $D$. 
\end{lemma}

\begin{remark}
When the representation of $g$ on the complex tangent space of $x$ has determinant $\neq 1$, i.e.,
when $m_1+m_2\neq 0\pmod{m}$, the assumption that $x\in\alpha^{-1}(0)$ for all $r>0$ is 
automatically satisfied. This is because when $m_1+m_2\neq 0\pmod{m}$,
the representation of $g$ on the fiber of the $G$-equivariant canonical bundle $K_\omega$ at $x$ is 
non-trivial, so that $\alpha(x)=0$ has to be true since $\alpha$ is fixed by $g$.

\end{remark}

The following recipe will be used frequently in determining the local 
intersection number $\text{int}_D(x)$: suppose a branch of $\cup_i C_i$
near $x$ is parametrized by a holomorphic map over a neighborhood of 
$0\in\C$ which is given in local coordinates by
$$
z_1=z^l,\; z_2=cz^{l^\prime}+ \cdots \mbox{ (higher order terms)}, 
$$
where $l^\prime>l$ unless $l=1$ and $c=0$, and suppose the multiplicity 
of the branch is $n$, then the contribution of the branch to $\text{int}_D(x)$
is equal to $nl$ provided that the $J$-holomorphic disk $D$ is not tangent
to $z_2=0$ at $x$ (i.e., the tangent space of the branch at $x$). See
Theorem 7.1 in Micallef and White \cite{MW}.

\vspace{3mm}

We end with a few remarks about the case of $b_2^{+}=1$ (for simplicity we assume the manifold 
is simply connected or is at least of $b_1=0$). Note that, in particular,  a simply connected, 
symplectic $4$-manifold with $b_2^{+}=1$ is homeomorphic to a rational surface, and when the manifold 
does not admit a smooth circle action, the smooth structure must be an exotic one which is characterized 
by $c_1(K)\cdot [\omega]>0$ and $c_1(K)^2\geq 0$ assuming the minimality of the
manifold (cf. \cite{Li}). 

Our proof of the Main Theorem broke down in the case of $b_2^{+}=1$, even
though the main line of arguments continues to work in this case. The missing
ingredient is the equivariant version of Taubes' theorem, i.e., for any
$r>0$, the $r$-version of Taubes' perturbed Seiberg-Witten equations 
associated to the equivariant canonical bundle has a solution which is 
fixed under the group action. Notice that with the condition 
$c_1(K)\cdot [\omega]>0$, one can argue using the wall-crossing formula
that the (non-equivariant) $r$-version of Taubes' perturbed 
Seiberg-Witten equations associated to the square of the canonical bundle
has a solution for sufficiently large $r>0$ provided that the dimension 
of the corresponding Seiberg-Witten moduli space is non-negative (which is
equivalent to $c_1(K)^2\geq 0$). One could argue similarly using wall-crossing
to get an equivariant version of this result which would be a good 
substitute of Taubes' theorem for our purpose, but unfortunitely the 
non-negativity of the dimension of the corresponding moduli space of 
equivariant Seiberg-Witten equations is much harder to come by; the 
calculation of the dimension requires knowledge about the group action near 
the fixed point set, which is not known a priori in general except for the 
case of a homology $\C\P^2$ due to the work of Edmonds and Ewing \cite{EE}.

\begin{theorem}
Let $X$ be a smooth $4$-manifold which is a homology $\C\P^2$. Then for
any symplectic structure $\omega$ with $c_1(K)\cdot [\omega]>0$, there
exists a constant $C>0$ such that there are no nontrivial symplectic 
$\Z_p$-actions of prime order on $X$ for $p>C$.
\end{theorem}  

\begin{remark}
(1) Theorem 1.8 holds true more generally when $X$ is only a $\Q$-homology
$\C\P^2$ provided that the $\Z_p$-actions are pseudofree, i.e., having only 
isolated fixed points. The primary examples of a symplectic $\Q$-homology
$\C\P^2$ with $c_1(K)\cdot [\omega]>0$ are complex surfaces of general type
with $p_g=0$ and $c_1^2=9$, known as fake projective planes. Many of them
have nontrivial automorphism groups which always give pseudofree actions, 
cf. \cite{Km}. 

(2) Since $b_2=b_2^{+}=1$ in this case, one can always rescale $\omega$
so that $[\omega]$ is a generater of $H^2(X)$. With this choice of $\omega$,
$N_\omega=1$ and $c_1(K)\cdot [\omega]=3$, so that the constant $C$ in
Theorem 1.8 is in fact independent of $\omega$. It is an interesting 
problem to find out the optimal value of $C$ in Theorem 1.8.
\end{remark}     

An earlier version of this paper, under a slightly different title `` On the orders of 
periodic symplectomorphisms of $4$-manifolds'', has appeared as MPIM-Leipzig Preprint no. 30/2009.

\vspace{3mm}

The organization of this paper is as follows. Section 2 consists of 
a set of preliminary lemmas preparing for the proof of the Main Theorem. 
In particular, it contains the proof of Lemma 1.6. Section 3 is devoted 
to the proof of the Main Theorem. The proof of Theorem 1.8 is given 
in Section 4, and the proof of Theorem 1.1 is given in Section 5.

\section{Preliminary lemmas} 

\begin{lemma}
Let $M$ be a compact $n$-manifold with a continuous $\Z_p$-action. For any $0\leq k\leq n$,
the induced action on $H_k(M;\Q)$ must be trivial if $p>1+b_k$, where $b_k$ is the k-th Betti
number of $M$. 
\end{lemma}

\begin{proof}
The $\Z_p$-action induces an integral $\Z_p$-representation on the free part of $H_k(M)$,
i.e., $H_k(M)/Tor H_k(M)$.  Then $H_k(M)/Tor H_k(M)$ as an integral $\Z_p$-representation 
is isomorphic to a direct sum 
$$
\Z[\Z_p]^r\oplus \Z^t\oplus \Z[\mu_p]^s \oplus J^e
$$
for some integers $r,t,s,e\geq 0$, where the group ring $\Z[\Z_p]$ is the regular representation of 
$\Z$-rank $p$, $\Z$ is the trivial representation of $\Z$-rank $1$, $\Z[\mu_p]$ is the representation 
of cyclotomic type of $\Z$-rank $p-1$, which is the kernel of the augmentation homomorphism 
$\Z[\Z_p]\rightarrow \Z$ (here $\mu_p\equiv \exp(2\pi i/p)$), and $J$ is some non-zero 
ideal of $\Z[\mu_p]$ which is also of $\Z$-rank $p-1$. See Curtis-Reiner \cite{CR}, 
Theorem (74.3) in p. 508. This gives 
$$
b_k=rp+t+(s+e)(p-1). 
$$
Now if $p>1+b_k$, the only solution to the above equation is $r=s=e=0$ and $t=b_k$,
which means that the induced action of $\Z_p$ on $H_k(M)/Tor H_k(M)$ is trivial. 
\end{proof}

\begin{lemma}
Let $G$ be a finite group acting on a symplectic $4$-manifold $(M,\omega)$, 
preserving the form $\omega$ and inducing a trivial action on $H_2(M;\Q)$. 
Then there exists a symplectic $4$-manifold $(M^\prime,\omega^\prime)$,
which is a symplectic blowdown of $(M,\omega)$, with an induced $G$-action 
preserving the form $\omega^\prime$ and inducing a trivial action on 
$H_2(M^\prime;\Q)$, such that for any $G$-equivariant, $\omega^\prime$-compatible
almost complex structure $J$, $M^\prime$ contains no embedded $J$-holomorphic 
$2$-spheres with self-intersection $-1$. Furthermore, if $[\omega]$ is 
rational, so is $[\omega^\prime]$, and one has 
$C_{\omega^\prime}\leq C_\omega$.
\end{lemma}

\begin{proof}
Suppose there exists a $G$-equivariant, $\omega$-compatible almost 
complex structure $J$ on $M$ such that $M$ contains an embedded $J$-holomorphic 
$2$-sphere $C$ with $C^2=-1$. Since $G$ acts trivially on $H_2(M;\Q)$,  the class of 
$g\cdot C$ equals the class of $C$ for any $g\in G$. This implies $g\cdot C=C$, i.e., $C$ 
is invariant under $G$, because otherwise 
$(g\cdot C)\cdot C\geq 0$ by the positivity of intersections of 
$J$-holomorphic curves, which
then contradicts the identity $(g\cdot C)\cdot C=C^2=-1$. 

We $G$-equivariantly blow down $(M,\omega)$ along $C$ and obtain a 
symplectic $4$-manifold
$(M^\prime,\omega^\prime)$, which inherits a symplectic $G$-action 
from $(M,\omega)$. Clearly
the induced action of $G$ on $H_2(M^\prime;\Q)$ is also trivial. 

We shall prove that if $[\omega]$ is a rational class, so is 
$[\omega^\prime]$, and one has 
$C_{\omega^\prime}\leq C_\omega$. To see this, let $x\in M^\prime$ be 
the image of $C$
under the blowing down. Then there exist small neighborhoods $U$ of $C$ 
in $M$ and $U^\prime$
of $x$ in $M^\prime$, such that $(M\setminus U,\omega)$ and 
$(M^\prime\setminus U^\prime,
\omega^\prime)$ are symplectomorphic, cf. \cite{MS}. It follows easily 
that $[N_\omega\omega^\prime]$
is an integral class. (Here recall that $N_\omega$ 
is the smallest positive integer
such that $[N_\omega\omega]$ is an integral class.) This 
implies $N_{\omega^\prime}\leq N_\omega$, and since 
$c_1(K_{M^\prime})\cdot [\omega^\prime]<c_1(K_{M})\cdot [\omega]$,
one clearly has $C_{\omega^\prime}\leq C_\omega$. 

This process will terminate because $b_2(M^\prime)=b_2(M)-1$. At the 
end we obtain a symplectic $4$-manifold, still denoted by 
$(M^\prime,\omega^\prime)$, such that for any $G$-equivariant,
$\omega^\prime$-compatible almost complex structure $J$,  $M^\prime$ 
contains no embedded 
$J$-holomorphic $2$-spheres with self-intersection $-1$. 
\end{proof}

A natural question is whether the manifold $(M^\prime,\omega^\prime)$ is
minimal. Recall that a symplectic $4$-manifold is said to be minimal if there 
exist no embedded symplectic $2$-spheres with self-intersection $-1$. 
The next lemma gives an answer to this question.

\begin{lemma}
Let $(M,\omega)$ be a symplectic $4$-manifold which is not rational nor
ruled. If there exists a $\omega$-compatible almost complex structure
$J_0$ such that $M$ contains no embedded $J_0$-holomorphic $2$-spheres 
with self-intersection $-1$, then $(M,\omega)$ is minimal. 
\end{lemma}

\begin{proof}
Suppose to the contrary that $(M,\omega)$ is not minimal, and let 
$\Sigma$ be an embedded symplectic $2$-sphere with self-intersection 
$-1$ in $M$. Then there 
exists a $\omega$-compatible almost complex structure $J^\prime$ on $M$ 
such that $\Sigma$ is $J^\prime$-holomorphic. On the other hand, 
by Corollary 3.3.4 in McDuff-Salamon \cite{MS1}, for any $J$ an
embedded $J$-holomorphic $2$-sphere with self-intersection $-1$ 
is always a regular point in the corresponding moduli space of 
$J$-holomorphic curves, which is also the only point in the moduli
space because of positivity of intersections of $J$-holomorphic 
curves. This implies that the Gromov invariant counting 
$J$-holomorphic $2$-spheres in the class of $\Sigma$ equals $\pm 1$. 
In particular, there exists a finite set of $J_0$-holomorphic curves 
$\{\Gamma_j\}$ such that 
$$
[\Sigma]=\sum_j l_j\Gamma_j \mbox{ for some integers } l_j>0.
$$

Consider first the case where $b_2^{+}>1$. By Taubes \cite{T} the 
Gromov invariant $Gr(K_M)\neq 0$, so that there exists a finite set
of $J_0$-holomorphic curves $\{C_i\}$ such that 
$$
c_1(K_M)=\sum_i n_i C_i \mbox{ for some integers } n_i>0. 
$$
Since $c_1(K_M)\cdot [\Sigma]=-1<0$, there exists a $j$ such that 
$c_1(K_M)\cdot \Gamma_j<0$. Because of positivity of intersections of 
pseudo-holomorphic curves, there must be an $i$ such that $\Gamma_j=C_i$. 
Now $c_1(K_M)\cdot C_i<0$ implies that 
$C_i^2\leq \frac{1}{n_i}c_1(K_M)\cdot C_i<0$,
and by the adjunction inequality one has 
$$
0\leq \text{genus}(C_i)\leq C_i^2+c_1(K_M)\cdot C_i+2\leq (-1)+(-1)+2=0,
$$
which implies that $C_i$ is an embedded $2$-sphere with self-intersection 
$-1$, contradicting the assumption on $J_0$. Hence $(M,\omega)$ must be 
minimal in this case.

The case of $b_2^{+}=1$ is slightly more involved. First, we symplectically
blow down $(M,\omega)$ to get a minimal symplectic $4$-manifold 
$(M^\prime,\omega^\prime)$. Since $(M,\omega)$ is not rational 
nor ruled,  either $Gr(K_{M^\prime})\neq 0$ or $Gr(2K_{M^\prime})\neq 0$ 
(cf. Prop 5.2 in \cite{LL1} and Prop. 4.1 in \cite{Li}). On the other hand, by the blowup formula 
of Gromov invariant \cite{LL}, if we denote by $E_1, \cdots, E_n\in H^2(M)$ 
the exceptional divisors of the symplectic blowdown 
$\pi:M\rightarrow M^\prime$ (i.e., the Poincar\'{e} duals of the symplectic
$(-1)$-spheres in $M$), then in the former case
$$
Gr(K_M)=Gr(\pi^\ast(K_{M^\prime})-\sum_{s=1}^n E_s)=
Gr(K_{M^\prime})\neq 0,
$$
which implies that $(M,\omega)$ is minimal as we argued for $b^{+}_2>1$. In the latter 
case, we have 
$$
Gr(2K_M+\sum_{s=1}^n E_s)=Gr(\pi^\ast(2K_{M^\prime})-\sum_{s=1}^n E_s)=
Gr(2K_{M^\prime})\neq 0,
$$
which implies that there exists a finite set of $J_0$-holomorphic curves 
$\{\hat{C}_k\}$ such that 
$$
c_1(2K_M+\sum_{s=1}^n E_s)=\sum_k \hat{n}_k \hat{C}_k 
\mbox{ for some integers } \hat{n}_k>0. 
$$
Now observe that the Gromov invariant $Gr(-E_s)\neq 0$ for each $s$.
(Note that we have proved this fact for the Poincar\'{e} dual of $\Sigma$.) 
Hence for each $s$, there exists a finite set of $J_0$-holomorphic curves 
$\{\Gamma_{js}\}$ such that 
$$
c_1(-E_s)=\sum_j l_{js} \Gamma_{js} \mbox{ for some integers } l_{js}>0.
$$
Putting these together, one has 
$$
c_1(2K_M)=\sum_k \hat{n}_k \hat{C}_k +\sum_{j,s} l_{js} \Gamma_{js} 
=\sum_i n_i C_i
$$
for a finite set of $J_0$-holomorphic curves $\{C_i\}$ and integers
$n_i>0$. Again, since $c_1(2K_M)\cdot [\Sigma]=-2<0$, there exists 
a $j$ such that $c_1(2K_M)\cdot \Gamma_j<0$. Then there must be an $i$ 
such that $\Gamma_j=C_i$, and as we argued earlier, $c_1(K_M)\cdot C_i<0$ 
implies that $C_i$ is an embedded $2$-sphere with self-intersection $-1$, 
contradicting the assumption on $J_0$. This proves that $(M,\omega)$ 
is also minimal in the case of $b_2^{+}=1$. 
\end{proof}

We remark that Lemma 2.3 is false if one drops the assumption 
that $(M,\omega)$ 
is not rational nor ruled as shown by the following example: the Hirzebruch 
surface $F_3$ contains no $(-1)$-holomorphic curves but it is not minimal
as a symplectic $4$-manifold. (Thanks to Tian-Jun Li for pointing out an
error in the original version of the paper and communicating this example 
to me.)

Let $G\equiv \Z_p$. With the preceding understood
(i.e., Lemma 2.1, Lemma 2.2 and Lemma 2.3), it suffices, for the proof of the Main
Theorem, to consider the following simplified version: $(X,\omega)$ is minimal and 
the $\Z_p$-action is homologically trivial over $\Q$. Moreover, it suffices to show that
there is a universal constant $c>0$ such that $p$ must be bounded from above by 
$$
C\equiv c (1+c_1(K)^2)^2 C_\omega^2.
$$
Indeed, given any $(X,\omega)$, we assume $p>5(1+b_1(X)+b_2(X))$ (note that, particularly, 
$p>5$). Then by Lemma 2.1, the $\Z_p$-action 
is trivial on $H_k(X;\Q)$ for any $0\leq k\leq 4$. Furthermore, by Lemma 2.2 and Lemma 2.3, 
one can $G$-equivariantly blow down $X$ until becoming minimal. Call the resulting 
symplectic $4$-manifold $(X^\prime,\omega^\prime)$. Then
$$
c_1(K_{X^\prime})^2\leq c_1(K_X)^2+b_2(X). 
$$
Our claim follows from the fact that $c_1(K_X)^2\leq 10(1+b_1(X)+b_2(X))$ and that 
$C_{\omega^\prime}\leq C_\omega$ (cf. Lemma 2.2). 

Now since $b_G^{2,+}=b_2^{+}>1$, the equivariant version of Taubes' theorem in
\cite{T} applies here, and as we have explained earlier in Section 1, 
after fixing a generic choice of 
$G$-equivariant, $\omega$-compatible almost complex structure $J$, 
there is a finite set of 
$J$-holomorphic curves $\{C_i\}$, such that $c_1(K)=\sum_i n_i C_i$ 
for some integers $n_i>0$.
Furthermore, the set $\cup_i C_i$ is $G$-invariant, and $\cup_i C_i$ 
contains all the fixed points 
of $G$ except for those isolated ones at which the representation 
of $G$ on the complex tangent 
space has determinant $1$. Notice that since $(X,\omega)$ is minimal, 
$c_1(K)\cdot C_i\geq 0$ and $c_1(K)\cdot C_i\geq C_i^2$ for any $i$.
(We remark that by the equivariant symplectic neighborhood theorem, 
$J$ may be taken to be integrable in a neighborhood of the fixed point set, 
even though $J$ has to be chosen to be generic in order to rule out certain 
possibilities.) 

In the next four lemmas (i.e., Lemmas 2.4-2.7), we assume that the $G$-action is
pseudofree. This means that none of the $J$-holomorphic curves $C_i$ is fixed under the action.
(In fact the non-pseudofree case can be easily eliminated, cf. Lemma 3.2.)

\begin{lemma}
(1) If there exists a $C_i$ with $\text{genus}(C_i)\geq 2$, then $p\leq 82 c_1(K)^2$.

(2) If there exists a $C_i$ with $\text{genus}(C_i)=1$ and $c_1(K)\cdot C_i\geq 1$, 
then $p\leq c_1(K)^2$.
\end{lemma}

\begin{proof}
(1) First consider the case where $C_i$ is invariant under $G$. By the assumption that the $G$-action 
is pseudofree, the induced $G$-action on $C_i$ is nontrivial. 
By Hurwitz's theorem, $p\leq 82(\text{genus}(C_i)-1)$. On the other hand, by the adjunction
inequality, 
$$
\text{genus}(C_i)-1\leq \frac{1}{2}(C_i^2+c_1(K)\cdot C_i)\leq c_1(K)\cdot C_i\leq c_1(K)^2. 
$$
Hence $p\leq 82 c_1(K)^2$ as claimed. 

Now suppose $C_i$ is not invariant under $G$. Then $g\cdot C_i\neq C_i$ for all $1\neq g\in G$.
This implies that
$$
c_1(K)^2\geq c_1(K)\cdot (\sum_{g\in G} g\cdot C_i)=p\cdot c_1(K)\cdot C_i\geq p,
$$
because $c_1(K)\cdot C_i\geq  \frac{1}{2}(C_i^2+c_1(K)\cdot C_i)\geq \text{genus}(C_i)-1\geq 1$. 
The lemma follows. 

(2) First consider the case where $C_i$ is invariant under $G$. 
Again by the pseudofree assumption, the induced $G$-action on $C_i$ is nontrivial. 
Since we assume $p>5$, and by assumption $\text{genus}(C_i)=1$, we see immediately 
that $C_i$ contains no fixed points of $G$. With this understood, the dimension of the moduli
space of the corresponding $G$-invariant $J$-holomorphic curves at $C_i$, which is given by
$$
2(-\frac{1}{p}c_1(K)\cdot C_i+2(1-\text{genus}(C_i/G)))=-\frac{2}{p}c_1(K)\cdot C_i,
$$
is negative because of the assumption $c_1(K)\cdot C_i\geq 1$. Note that here since the 
$G$-action on $C_i$ is free, $\text{genus}(C_i/G)=1$. By choosing a generic $G$-equivariant $J$, 
this case can be ruled out. (Here and throughout the rest of the paper, moduli spaces of $G$-invariant
pseudo-holomorphic curves in $X$ are canonically identified with the corresponding moduli spaces of
pseudo-holomorphic curves in the orbifold $X/G$.)

Suppose $C_i$ is not invariant under $G$. Then
$$
c_1(K)^2\geq c_1(K)\cdot (\sum_{g\in G} g\cdot C_i)=p\cdot c_1(K)\cdot C_i\geq p
$$
as claimed. This finishes the proof of the lemma. 
\end{proof}

\begin{lemma}
For any $i$, if $\text{genus}(C_i)=1$ and $c_1(K)\cdot C_i=0$, then $C_i$ is an embedded 
torus of self-intersection $0$, which is disjoint from the rest of the set of $J$-holomorphic 
curves $\{C_i\}$. 
\end{lemma}

\begin{proof}
Since $0=c_1(K)\cdot C_i\geq C_i^2$, we have, by the adjunction inequality, that 
$$
0\geq C_i^2+c_1(K)\cdot C_i\geq \text{genus}(C_i)-1=0,
$$
which implies that $C_i^2=c_1(K)\cdot C_i=0$, and $C_i$ is embedded. If $C_i$ is not
disjoint from the rest of the set of $J$-holomorphic curves $\{C_i\}$, we would have 
$c_1(K)\cdot C_i> C_i^2=0$, which is a contradiction. Hence the lemma. 
\end{proof}

Since we assume $p>5$, the curves $C_i$ as described in Lemma 2.5 will not contain any
fixed points of $G$. Furthermore, they will not make any contributions to the calculation
of either $c_1(K)\cdot C_i$ for any $i$, or $c_1(K)^2$. So this kind of $J$-holomorphic curves will play no
role in our argument, and henceforth for simplicity we simply assume they do not exist. 

\begin{lemma}
If there exists a $C_i$ with $\text{genus}(C_i)=0$ which is not invariant under $G$, then
$p\leq c_1(K)^2$. 
\end{lemma}

\begin{proof}
Since $C_i$ is not invariant under $G$, $g\cdot C_i\neq C_i$ for any $1\neq g\in G$.
On the other hand, $G$ acts trivially on $H_2(X;\Q)$, so that $C_i^2=(g\cdot C_i)\cdot C_i\geq 0$.

If $c_1(K)\cdot C_i\geq 1$, we have as before that 
$$
c_1(K)^2\geq c_1(K)\cdot (\sum_{g\in G} g\cdot C_i)=p\cdot c_1(K)\cdot C_i\geq p.
$$
If  $c_1(K)\cdot C_i=0$, then $C_i^2=0$ as well, which implies that $(g\cdot C_i)\cdot C_i=0$
for all $g\in G$. In particular $g\cdot C_i$ and $C_i$ are disjoint for any $1\neq g\in G$, so that
$C_i$ contains no fixed points of $G$. The dimension of the moduli space of the corresponding 
$J$-holomorphic curves at $C_i$ is given by 
$$
d=2(-c_1(K)\cdot C_i+2(1-\text{genus}(C_i))-3)=-2,
$$
so that by choosing a generic $G$-equivariant $J$, such a $C_i$ does not exist. 
\end{proof}

\begin{lemma}
Suppose there exist $i,j$, $i\neq j$,  such that (i) $\text{genus}(C_i)=\text{genus}(C_i)=0$, (ii) 
$C_i$ and $C_j$ intersect at a point
which is not fixed under $G$. Then $p\leq 2+2c_1(K)^2$.
\end{lemma}

\begin{proof}
Suppose both of $C_i$, $C_j$ are invariant under $G$; otherwise the lemma follows from
the previous lemma. Without loss of generality, we assume $n_i\geq n_j$. Set $\delta=c_1(K)\cdot C_j$.
Then
$$
\delta\geq (n_iC_i+n_jC_j)\cdot C_j\geq n_j(C_i\cdot C_j+C_j^2)\geq n_j(p+C_j^2). 
$$
Here we used $C_i\cdot C_j\geq p$, which follows from the fact that the set $C_i\cup C_j$ is
$G$-invariant and $C_i$, $C_j$ intersect at a point not fixed under $G$. This gives 
$\delta\geq\frac{1}{n_j}\delta\geq p+C_j^2$. 

Now by the adjunction inequality, we obtain
$$
(\delta-p)+\delta +2\geq C_j^2+c_1(K)\cdot C_j+2\geq 0, 
$$
which gives rise to $p\leq 2+2\delta\leq 2+2c_1(K)^2$, as we claimed.
\end{proof}

We end this section with the 

\vspace{3mm}

\noindent{\bf Proof of Lemma 1.6}

\vspace{3mm}

The local intersection number $\text{int}_D(x)$ is defined to be the limit
$$
\text{int}_D(x)=\lim_{n\rightarrow\infty}\int_D \frac{i}{2\pi} F_{a_n}
$$
for a sequence of solutions $((\alpha_n,\beta_n),a_n)$ to the $r_n$-version of 
the Taubes' perturbed Seiberg-Witten equations, where $r_n\rightarrow\infty$
as $n\rightarrow\infty$, cf. Proposition 5.6 in \cite{T}. In Lemma 5.8 of \cite{T},
Taubes gave a lower bound for $\text{int}_D(x)$ which takes the form
$$
\int_D \frac{i}{2\pi} F_{a_n}\geq m_0+z_3(4^{-n}+\rho^2),
$$
where $\rho>0$ can be taken arbitrarily small, and $m_0$ is a positive integer
(see (5.19) and (5.20) in \cite{T}).  Here $z_3$ is an independent constant. Clearly,
$$
\text{int}_D(x)\geq m_0.
$$

To explain $m_0$, recall that by our assumption, $x\in\alpha_n^{-1}(0)$ for all $n$.
Fix a Gaussian coordinate system at $x$ and pull back the solutions $((\alpha_n,\beta_n),a_n)$
to the Gaussian system. After rescaling by a factor $\sqrt{r_n}$, the solutions 
converge in $C^\infty$-topology over compact subsets to a solution $((\alpha_0,0),a_0)$ to the
$r=1$ version of the Taubes' perturbed Seiberg-Witten equations on $\C^2$. Moreover, the 
$U(1)$-connection $a_0$ defines a holomorphic structure on the trivial complex line bundle over
$\C^2$ of which $\alpha_0$ is a holomorphic section. Finally, $\alpha_0^{-1}(0)$ is the 
zero set of a polynomial on $\C^2$. With the preceding understood, the number $m_0$ is the
local contribution at $0\in\C^2$ to the intersection number of a generic complex line in $\C^2$ 
with $\alpha_0^{-1}(0)$. 

Note that there is an induced complex linear action of $g$ on $\C^2$ as well as a corresponding 
action on the associated Seiberg-Witten equations on $\C^2$.  If we write $(z_1,z_2)$ for the complex coordinates on $\C^2$, then 
$$
g\cdot (z_1,z_2)=(\lambda^{m_1}z_1,\lambda^{m_2}z_2)
$$ 
where $\lambda=\exp(2\pi i/m)$ with $m\equiv \text{order}(g)$ and $m_1, m_2$ are the weights of
$g$ at $x$. 

Now write $\alpha_0=f(z_1,z_2)\cdot s$, where $s$ is a non-zero holomorphic section, and
$$
f(z_1,z_2)=\sum_{i=1}^N c_i z_1^{a_{1,i}}z_2^{a_{2,i}}+ \cdots \mbox{(higher order terms)}.
$$
Here $a\equiv a_{1,i}+a_{2,i}>0$ which is independent of $i=1,2,\cdots, N$. Then the above 
interpretation of $m_0$ shows that $m_0\geq a$. On the other hand, the representation of 
$g$ on the fiber of the $G$-equivariant canonical bundle is given by multiplication by
$\lambda^{-(m_1+m_2)}$, where $\lambda=\exp (2\pi i/m)$, and $m\equiv \text{order}(g)$.
Apparently $g\cdot s=\lambda^{-(m_1+m_2)}s$ and $g\cdot \alpha_0=\alpha_0$, which implies
that 
$$
f(g\cdot (z_1,z_2))=\lambda^{-(m_1+m_2)}\cdot f(z_1,z_2). 
$$
The above equation gives the congruence relation 
$$
a_{1,i}m_1+a_{2,i}m_2=-(m_1+m_2) \pmod{m}, \;\;\forall i=1,2,\cdots,N.
$$
The lemma follows easily by taking $(a_1,a_2)$ to be any of the $(a_{1,i},a_{2,i})$'s. 

\hfill $\Box$

\section{Proof of the Main Theorem}

Before we start, it is useful to make observation of the following fact.

\begin{lemma}
$\sum_i n_i\leq C_\omega$, where $c_1(K)=\sum_i n_iC_i$. In particular, $n_i\leq C_\omega$ for 
each $i$. 
\end{lemma}

\begin{proof}
$c_1(K)\cdot [\omega]=\sum_i n_i\omega(C_i)\geq \sum_i n_i\cdot\frac{1}{N_\omega}$, from
which the lemma follows. 
\end{proof}

The following lemma eliminates the case of non-pseudofree actions.

\begin{lemma}
If there exists a $C_i$ which is fixed under $G$. Then $p\leq 1+C_\omega$. 
\end{lemma}

\begin{proof}
Let $n_i$ be the multiplicity of $C_i$. We pick a point $x\in C_i$ such that $x$ does not
lie in any other $C_j\neq C_i$. Let $D$ be a $J$-holomorphic disk intersecting $C_i$ transversely
and $D\cap (\cup_j C_j)=\{x\}$. Then the local intersection number 
$$
\text{int}_D(x)=n_i,
$$
cf. \cite{T}, Section 5. Let $g\in G$ be the element whose action near $x$ is given $g\cdot (z_1,z_2)=
(z_1,\mu_p z_2)$, where $\mu_p\equiv \exp(2\pi i/p)$. Then by Lemma 1.6, there exist non-negative
integers $a_1,a_2$ satisfying $(a_1+1)\cdot 0+(a_2+1)\cdot 1\equiv 0\pmod{p}$ (here $m_1=0$, $m_2=1$),
such that $\text{int}_D(x)\geq a_1+a_2$. It follows that $n_i=\text{int}_D(x)\geq a_2\geq p-1$.  
This gives
$$
p\leq 1+n_i\leq 1+C_\omega. 
$$
\end{proof}

We shall assume, in what follows, that the $\Z_p$-action is pseudofree. 
With the preliminary lemmas proved in the previous section (i.e., Lemmas 2.4-2.7), 
we may assume without loss of
generality that the curves in the set $\{C_i\}$ satisfy:

\begin{itemize}
\item [{(1)}] each $C_i$ is a $2$-sphere, which may be singular or immersed;
\item [{(2)}] each $C_i$ is $G$-invariant, containing $\leq 2$ fixed points;
\item [{(3)}] two distinct $C_i$, $C_j$ intersect only at fixed points of $G$;
\item [{(4)}] each $C_i$ is embedded away from the fixed points of $G$ on $C_i$. 
\end{itemize}
Here for the last condition, (4), if there is a $C_i$ which is not embedded away from the fixed points
of $G$, we obtain a bound for $p$ by the adjunction inequality: let $y_k$, $k=1,2, \cdots, p$, be a subset
of singular points of $C_i$ which is invariant under $G$ and denote by $\delta_{y_k}$ the contribution
of $y_k$ to the adjunction inequality, then $\delta_{y_k}\geq 2$ and 
$$
c_1(K)^2+1\geq c_1(K)\cdot C_i+1\geq \frac{1}{2}(C_i^2+c_1(K)\cdot C_i)+1\geq 
\frac{1}{2}\sum_{k=1}^p \delta_{y_k}\geq p.
$$

\vspace{3mm}

\noindent{\bf Case (a)}: $c_1(K)^2=0$. First of all, notice that the $\Z_p$-action must have a fixed 
point. This is because if the original symplectic $4$-manifold is not minimal, then after blowing down 
the induced action must have a fixed point, which is the image of the exceptional divisor under the
blowing down. If the original manifold is minimal, then by the assumption we made in the Main Theorem,
the Euler characteristic of the manifold must be non-zero. In any event, the $\Z_p$-action has a fixed point.

With this understood, according to Corollary B of \cite{CK}, the $\Z_p$-action must be trivial unless
$p=1\pmod{4}$ or $p=1\pmod{6}$. Moreover, from the proof of Corollary B, the following are also 
true: (i) when $p=1\pmod{4}$, there must be $C_i,C_j$, both embedded, with $n_i=n_j$, such that $C_i^2=C_j^2=-2$ 
and $C_i$, $C_j$ intersect at a fixed point $x$ with tangency of order $2$; (ii) when $p=1\pmod{6}$, 
then either there is a $C_i$ which is a $2$-sphere with a cusp singularity $x$ fixed by $G$, or 
there are $3$ distinct embedded $(-2)$-spheres $C_i, C_j, C_k$ intersecting transversely at a fixed point 
$x$ of $G$. Furthermore, there are no fixed points where
the representation of $G$ on the complex tangent space has determinant $1$. 
By Remark 1.7, Lemma 1.6 applies here to all the fixed points of $G$. We fix
a $J$-holomorphic disk $D$ whose tangent plane at $x$ is different from that
of any of the $J$-holomorphic curves in $\{C_i\}$. In this case, 
$\text{int}_D(x)$ can be easily determined using Theorem 7.1 of Micallef and
White \cite{MW}. 

In case (i), Lemma 1.6 gives us
$$
C_\omega \geq n_i+n_j=\text{int}_D(x)\geq a_1+a_2,
$$
where $a_1,a_2$ satisfy $(a_1+1)\cdot 1+(a_2+1)\cdot 2\equiv 0\pmod{p}$.  This implies that
$$
p\leq 2(a_1+a_2)+3\leq 2C_\omega+3. 
$$
As for case (ii), in the former case of a cusp sphere, Lemma 1.6 gives us
$$
2C_\omega \geq 2n_i=\text{int}_D(x)\geq a_1+a_2,
$$
where $a_1,a_2$ satisfy $(a_1+1)\cdot 2+(a_2+1)\cdot 3\equiv 0\pmod{p}$.  This implies that
$$
p\leq 3(a_1+a_2)+5\leq 6C_\omega+5. 
$$
In the latter case of (ii), Lemma 1.6 gives us
$$
C_\omega \geq n_i+n_j+n_k=\text{int}_D(x)\geq a_1+a_2,
$$
where $a_1,a_2$ satisfy $(a_1+1)\cdot 1+(a_2+1)\cdot 1\equiv 0\pmod{p}$.  This implies that
$$
p\leq (a_1+a_2)+2\leq C_\omega+2. 
$$
The proof of the Main Theorem for the case where $c_1(K)^2=0$ follows. 

\vspace{3mm}

\noindent{\bf Case (b)}: $c_1(K)^2>0$.  We start with the following lemma.

\begin{lemma}
By choosing a generic $G$-equivariant almost complex structure $J$, 
the set $\cup_i C_i$ contains no fixed points 
of $G$ where the representation of $G$ on the complex tangent space has determinant $1$. 
\end{lemma}

\begin{proof}
Suppose $x$ is such a fixed point, and $x\in C_0\in\{C_i\}$. 
Let $f:\s^2\rightarrow X$ be a $J$-holomorphic 
map parametrizing $C_0$, and let $t_1, t_2\in\s^2$ be the two points mapped to fixed points under
$f$ such that $f(t_1)=x$. Note that $t_1,t_2$ are fixed under the induced action of $G$ on $\s^2$.
Let $g_1,g_2\in G$ be the elements which act by a rotation of angle $2\pi/p$
near $t_1,t_2$ respectively. Moreover, suppose the actions of $g_1,g_2$ near the fixed points in $X$ 
are given respectively by 
$$
g_i \cdot (z_1,z_2)=(\mu_p^{m_{i,1}}z_1,\mu_p^{m_{i,2}} z_2), \; i=1,2, 
$$
where $\mu_p=\exp(2\pi i/p), 0<m_{i,1}, m_{i,2}<p$. 
Then the dimension of the moduli space of the corresponding $G$-invariant $J$-holomorphic curves 
at $C_0$ is 
\begin{eqnarray*}
d & = & 2(-\frac{1}{p} c_1(K)\cdot C_0+2-\sum_{i=1}^2\frac{m_{i,1}+m_{i,2}}{p}-1)\\
& = & -2(\frac{1}{p}c_1(K)\cdot C_0+\frac{m_{2,1}+m_{2,2}}{p}),\\
\end{eqnarray*}
see \cite{C1}, p. 19.  Here we used the fact that the representation of $G$ 
on the complex tangent space of $x$ has determinant $1$, so that $m_{1,1}+m_{1,2}=p$.
By choosing a generic $G$-equivariant $J$ (cf. \cite{C1}, Lemma 1.10), 
$d\geq 0$ if $C_0$ exists. But this is impossible because $c_1(K)\cdot C_0\geq 0$. 
\end{proof}

With the preceding lemma, Lemma 1.6 applies to any fixed point contained in $\cup_i C_i$
(cf. Remark 1.7). 

\begin{lemma}
If there exists a $C_i$ which is not embedded, then 
$$
p\leq 4C_\omega^2(3+2c_1(K)^2)^2.
$$ 
\end{lemma}

\begin{proof}
We first note that all $C_i$ are embedded away from the fixed points of $G$. Fix any curve $C_0$ in
the set $\{C_i\}$. Let $x\in C_0$ be a fixed point of $G$. We parametrize $C_0$ by a $J$-holomorphic
map $f_0:\s^2\rightarrow X$, and suppose $0\in\s^2$ is mapped to $x$ under $f_0$. In a local complex
coordinate system $(z_1,z_2)$ centered at $x$, suppose $f_0$ is represented by a holomorphic map with
$z$ as a local coordinate centered at $0\in \s^2$:
$$
z_1=z^{l_0}, z_2=c_0z^{l_0^\prime}+\cdots \mbox{(higher order terms)}, 
$$
where $l_0<l_0^\prime$ unless $c_0=0$ and $l_0=1$. 

We first show that if $l_0\geq 2$, then 
$$
p\leq \text{max} (16C_\omega^2, (5+2c_1(K)^2)^2, 4C_\omega^2(3+2c_1(K)^2)^2)=
4C_\omega^2(3+2c_1(K)^2)^2.
$$
Let $f_j: z\mapsto (z^{l_j}, c_j z^{l_j^\prime} +\cdots)$,
$j=1,2,\cdots, N$, be the holomorphic maps which parametrize all the branches of $\cup_i C_i$ 
near $x$ other than the one parametrized by $f_0$ in a neighborhood of $0\in\s^2$. Here for each $j$,
$l_j<l_j^\prime$ unless $c_j=0$ and $l_j=1$. If we fix a generator $g\in G$ and suppose the 
action of $g$ near $x$ is given by
$g\cdot (z_1,z_2)=(\mu_p^mz_1,\mu_p^{m^\prime}z_2)$, where $\mu_p=\exp(2\pi i/p)$ and 
$0<m,m^\prime<p$, then it follows easily that $l_j=k_jm, l_j^\prime=k_jm^\prime \pmod{p}$ for 
some $k_j$ for all $0\leq j\leq N$. 

We assume $p\geq 16C_\omega^2$. There are two possibilities: 

Case (i): $l_j\leq (2C_\omega)^{-1}\cdot \sqrt{p}$ for all $j=0,1,2,\cdots,N$. Denote by $n_j^\prime$ the
multiplicity of the branch parametrized by the map $f_j$, $j=0,1,2,\cdots, N$. Then by Lemma 3.1
we have $\sum_{j=0}^N n_j^\prime\leq 2C_\omega$. (Note that at most $2$ branches lie in the same 
$C_i$.) We obtain 
$$
\sqrt{p}=2C_\omega\cdot ((2C_\omega)^{-1}\cdot \sqrt{p})\geq (\sum_{j=0}^N n_j^\prime)\cdot
((2C_\omega)^{-1}\cdot \sqrt{p})\geq \sum_{j=0}^N n_j^\prime l_j. 
$$
Now if we pick a $J$-holomorphic disk $D$ whose tangent plane at $x$ is different from
that of any of the branches 
parametrized by $f_j$, $0\leq j\leq N$, then by Theorem 7.1 in Micallef and White \cite{MW}, 
$\text{int}_D(x)=\sum_{j=0}^N n_j^\prime l_j$. By Lemma 1.6, we obtain
$$
\sqrt{p}\geq \text{int}_D(x)\geq a_1+a_2,
$$
where $a_1,a_2$ satisfy $(a_1+1)l_j+(a_2+1)l_j^\prime\equiv 0\pmod{p}$, $0\leq j\leq N$, because of
the congruence relations $l_j=k_jm, l_j^\prime=k_jm^\prime \pmod{p}$ for some $k_j$ for 
all $0\leq j\leq N$. Particularly, we have 
$$
(\sqrt{p}+2) l_0^\prime\geq (a_1+a_2 +2)l_0^\prime \geq (a_1+1)l_0+(a_2+1)l_0^\prime\geq p,
$$
which implies $l_0^\prime\geq (\sqrt{p}+2)^{-1}p\geq \sqrt{p}-2$. On the other hand, by 
Theorem 7.3 in  Micallef and White \cite{MW}, the point $x\in C_0$ makes a local contribution of
$\delta_x\geq (l_0-1)(l_0^\prime-1)$ to the adjunction inequality for $C_0$, which gives
$$
2c_1(K)^2\geq C_0^2+c_1(K)\cdot C_0\geq -2+ (l_0-1)(l_0^\prime-1)\geq 
-2+(\sqrt{p}-3). 
$$
Note that here we used $l_0\geq 2$.  This implies that $p\leq (5+2c_1(K)^2)^2$. 

Case (ii): there exists a $j=0,1,2,\cdots,N$ such that $l_j\geq (2C_\omega)^{-1}\cdot \sqrt{p}$. 
Then $l_j\geq 2$ since $p\geq 16C_\omega^2$, and $l_j^\prime>l_j\geq (2C_\omega)^{-1}\cdot \sqrt{p}$ 
for that $j$. Let $C_i$ be the $J$-holomorphic curve which contains the branch parametrized by $f_j$ 
near $x$. Then $x\in C_i$ makes a local contribution of $\delta_x\geq (l_j-1)(l_j^\prime-1)$ to the 
adjunction inequality for $C_i$, which gives
$$
2c_1(K)^2\geq C_i^2+c_1(K)\cdot C_i\geq -2+ (l_j-1)(l_j^\prime-1)\geq 
-2+((2C_\omega)^{-1}\sqrt{p}-1). 
$$
This implies that $p\leq 4C_\omega^2(3+2c_1(K)^2)^2$. Hence if $l_0\geq 2$, one has 
$$
p\leq \text{max} (16C_\omega^2, (5+2c_1(K)^2)^2, 4C_\omega^2(3+2c_1(K)^2)^2)
=4C_\omega^2(3+2c_1(K)^2)^2.
$$

To finish the proof of the lemma, it remains to rule out the possibility that there is a $w\in \s^2$,
$w\neq 0\in\s^2$, such that $f_0(w)=f_0(0)=x$. Note that by the arguments in the previous paragraphs,
we may assume that $f_0$ is embedded near both $0$ and $w$. Consider first
the case where the tangent planes $(f_0)_\ast (T_0\s^2)$ and $(f_0)_\ast (T_w\s^2)$ intersect transversely 
at $x$. Suppose $g\in G$ is the element which acts near $0\in\s^2$ as rotation by an angle 
of $2\pi/p$. Then $g^{-1}$ acts near $w\in\s^2$ as rotation by an angle of $2\pi/p$. It follows 
that the action of $g$
near $x$ is given in local coordinates by $g\cdot (z_1,z_2)=(\mu_p z_1,\mu_p^{-1}z_2)$,
where $\mu_p=\exp(2\pi i/p)$. But this has been ruled out by Lemma 3.3. Now if $(f_0)_\ast (T_0\s^2)=
(f_0)_\ast (T_w\s^2)$, then $g=g^{-1}$ on $(f_0)_\ast (T_0\s^2)=(f_0)_\ast (T_w\s^2)$, which implies
that $p=2$. But we have assumed that $p>5$. 

This shows that if $C_0$ is not embedded near $x$, one has to have 
$$
p\leq 4C_\omega^2(3+2c_1(K)^2)^2.
$$
\end{proof}

For the rest of this section we assume $p>4C_\omega^2(3+2c_1(K)^2)^2$, so 
that all $C_i$ are embedded.

\begin{lemma}
For any fixed point $x$ of $G$, if there exist two distinct $J$-holomorphic curves $C_i,C_j$ from
the set $\{C_i\}$ such that $C_i$, $C_j$ intersect at $x$ non-transversely, then 
$$
p\leq (3+C_\omega)^2(c_1(K)^2+2).
$$
\end{lemma}

\begin{proof}
First of all, we shall prove that for any $1\neq g\in G$, 
if the action of $g$ near $x$ is given in local
coordinates by $g\cdot (z_1,z_2)=(\mu_p^{m_1}z_1,\mu_p^{m_2}z_2)$, where 
$\mu_p=\exp(2\pi i/p)$ 
and $0<m_1,m_2<p$, then 
$$
\max (m_1,m_2)\geq (3+C_\omega)^{-1} p. 
$$
To see this, if both $m_1,m_2$ are less than $(3+C_\omega)^{-1} p$, then by Lemma 3.1, Lemma 1.6,
\begin{eqnarray*}
(C_\omega+2)\cdot (3+C_\omega)^{-1} p & \geq & (\sum_i n_i +2)\cdot (3+C_\omega)^{-1} p\\
& \geq  & (\text{int}_D(x)+2)\cdot (3+C_\omega)^{-1} p\\
& \geq  & (a_1+1)m_1+(a_2+1)m_2\\
& \geq & p,
\end{eqnarray*}
which is a contradiction. Here $D$ is chosen such that it is not tangent to any of the curves 
in $\{C_i\}$ which contains $x$, and consequently, $\text{int}_D(x)\leq \sum_i n_i$ by Theorem 7.1 in Micallef and White \cite{MW} (notice that we have 
assumed that each $C_i$ is embedded).  

With the preceding understood, since $C_i$, $C_j$ intersect at $x$ non-transversely, there exist 
local coordinates $z_1,z_2$ centered at $x$, such that locally $C_i$ is given by $z_2=0$, and
$C_j$ is given by the graph of $z_2=z_1^m + \cdots \mbox{(higher order terms)}$.  Let $g\in G$
be the element which acts on $C_i$ by a rotation of angle $2\pi/p$ near $x$. Then the action of $g$ 
near $x$ is given by $g\cdot (z_1,z_2)=(\mu_p z_1,\mu_p^{m}z_2)$. We have just shown that
$$
m=\max (1,m)\geq (3+C_\omega)^{-1} p, 
$$
which implies that $C_i\cdot C_j\geq m\geq (3+C_\omega)^{-1} p$. 

Now we write $c_1(K)\cdot C_i=\sum_{k\neq i} n_k C_k\cdot C_i +n_i C_i^2$, and with the adjunction
inequality, we have
$$
\sum_{k\neq i} n_k C_k\cdot C_i +(n_i +1)C_i^2+2\geq 0. 
$$
This gives rise to
$$
C_i^2\geq -\frac{1}{n_i+1} (2+\sum_{k\neq i} n_k C_k\cdot C_i).
$$
Then we have 
\begin{eqnarray*}
c_1(K)\cdot C_i & = & \sum_{k\neq i} n_k C_k\cdot C_i +n_i C_i^2\\
                          & \geq & \sum_{k\neq i} n_k C_k\cdot C_i 
                          - \frac{n_i}{n_i+1} (2+\sum_{k\neq i} n_k C_k\cdot C_i)\\
                          &= & \frac{1}{n_i+1} (\sum_{k\neq i} n_k C_k\cdot C_i) -\frac{2n_i}{n_i+1}\\
                          &\geq &  \frac{1}{n_i+1}\cdot C_j\cdot C_i-\frac{2n_i}{n_i+1}\\
                          &\geq & \frac{1}{C_\omega+1}\cdot \frac{p}{3+C_\omega}-2.\\
\end{eqnarray*}
This implies that $p\leq (3+C_\omega)^2(c_1(K)\cdot C_i +2)\leq (3+C_\omega)^2(c_1(K)^2+2)$.
\end{proof}

\begin{corollary}
Suppose $p>(3+C_\omega)^2(c_1(K)^2+2)$. Then for any fixed point $x$,
there exist at most two distinct $C_i, C_j$ containing $x$. Moreover, $C_i, C_j$ 
intersect transversely at $x$.
\end{corollary}

\begin{proof}
Since $p>3+C_\omega$, we have $\max (m_1,m_2)\geq (3+C_\omega)^{-1} p>1$ for any $1\neq g\in G$
whose action is given in local coordinates by $g\cdot (z_1,z_2)=(\mu_p^{m_1} z_1,\mu_p^{m_2} z_2)$.
It follows that the action
of $G$ at $x$ has two distinct eigenvalues. If $x$ is contained in more than two distinct
$J$-holomorphic curves from the set $\{C_i\}$, there must be two distinct $C_i, C_j$ intersecting
non-transversely at $x$, which contradicts $p>(3+C_\omega)^2(c_1(K)^2+2)$. 
\end{proof}

With the preceding understood, we assume $p> (3+C_\omega)^2(c_1(K)^2+2)$.
Then for any $C_i$, there are $4$ possibilities:
\begin{itemize}
\item [{(1)}] $C_i$ does not intersect with any other curves in $\{C_i\}$;
\item [{(2)}] $C_i$ intersects with exactly one $C_j$ at exactly one fixed point;
\item [{(3)}] $C_i$ intersects with exactly one $C_j$ at two fixed points;
\item [{(4)}] $C_i$ intersects with two distinct $C_j, C_k$ at two fixed points.
\end{itemize}

Note that since $C_i$ is embedded, one has $c_1(K)\cdot C_i+C_i^2+2=0$. 

Case (1): $n_i C_i^2+ C_i^2+2=c_1(K)\cdot C_i+C_i^2 +2=0$, which implies 
$n_i=1$ and $C_i^2=-1$. This contradicts the minimality of $(X,\omega)$.

Case (2): $n_i C_i^2+n_j+C_i^2+2=c_1(K)\cdot C_i+C_i^2 +2=0$, which implies $C_i^2=-1$
if $n_j<n_i$. Hence in this case, one must have $n_j\geq n_i$ by the minimality of $(X,\omega)$.

Case (3):  $n_i C_i^2+2n_j+C_i^2+2=c_1(K)\cdot C_i+C_i^2 +2=0$, which implies $C_i^2=-1$
if $n_j<n_i$. By the symmetry between $i$ and $j$, we see that $n_i=n_j$, and 
$C_i^2=C_j^2=-2$. Moreover, $c_1(K)\cdot C_i=c_1(K)\cdot C_j=0$. 

Case (4): We assume that $n_i\geq n_j, n_k$. Then in this case,
$$
n_i C_i^2+n_j+n_k+C_i^2+2=c_1(K)\cdot C_i+C_i^2 +2=0,
$$
which implies that $n_i=n_j=n_k$ and $C_i^2=-2$. Moreover, $c_1(K)\cdot C_i=0$.

From the preceding analysis, it is easily seen that Case (2) can not occur, and that
for any $C_i$, $c_1(K)\cdot C_i=0$.
It follows that $c_1(K)^2=\sum_i n_i c_1(K)\cdot C_i=0$, which is a contradiction. 

This completes the proof of the Main Theorem. 

\section{Proof of Theorem 1.8}

\begin{lemma}
Let $(X,\omega)$ be a symplectic homology $\C\P^2$ with 
$c_1(K)\cdot [\omega]>0$. Then any symplectic $\Z_p$-action of prime
order on $X$ must be pseudofree and $p$ must be odd.
\end{lemma}

\begin{proof}
First of all, we show that a smooth involution on a homology $\C\P^2$
must have a $2$-dimensional component in the fixed point set (cf. \cite{Ed1}).
Suppose $g$ is a smooth involution which has only isolated fixed points.
Let $\Sigma$ be a smoothly embedded surface in $X$ which represents 
a generater of $H_2(X)$. By slightly perturbing $\Sigma$ we may assume 
that $\Sigma$ does not contain any fixed points of $g$, and furthermore,
$g\cdot \Sigma$ and $\Sigma$ intersect transversely. It is clear that
the intersection points of $g\cdot \Sigma$ and $\Sigma$ come in pairs,
so that the intersection number $(g\cdot\Sigma)\cdot \Sigma=0 \pmod{2}$.
However, since $\Sigma$ represents a generater of $H_2(X)$ for a homology
$\C\P^2$, $(g\cdot\Sigma)\cdot \Sigma=1 \pmod{2}$, which is a contradiction.

Secondly, we show that any symplectic $\Z_p$-action on $X$ must be pseudofree.
Suppose $Y$ is a $2$-dimensional component in the fixed point set. Then
since the action is naturally homologically trivial, $Y$ must be an
embedded $2$-sphere (cf. \cite{EE}), which is also naturally symplectic. From 
$Y\cdot [\omega]>0$ and the assumption that $c_1(K)\cdot [\omega]>0$,
we see that $c_1(K)\cdot Y>0$ also. But this violates the adjunction
inequality for $Y$ since we also have $Y^2>0$. Hence the lemma.
\end{proof}

With the preceding lemma, the following theorem of Edmonds and Ewing
will play a crucial role in the proof of Theorem 1.8.

\begin{theorem}
{\em(}Edmonds and Ewing, \cite{EE}{\em)} 
The fixed point set structure of a locally
linear, pseudofree, topological $\Z_p$-action of odd order on a homology
$\C\P^2$ is the same as that of a linear action on $\C\P^2$.
\end{theorem}

More concretely, a locally linear, pseudofree, topological $\Z_p$-action
of odd order has three fixed points $x_1,x_2,x_3\in X$. Fix a generater $g$
of the group. Then at each $x_i$, there is a pair of integers $(a_i,b_i)$ 
(unordered) satisfying $0<a_i,b_i<p$, such that the induced representation 
of $g$ on the tangent space at $x_i$ is given by
$$
(z_1,z_2)\mapsto (\mu_p^{a_i}z_1,\mu_p^{b_i}z_2), \mbox{ where }
\mu_p=\exp(2\pi i/p),
$$
for some complex structure on the tangent space which is compatible with
the orientation. Note that $(a_i,b_i)$ is unique up to a change 
of sign, i.e., a change to $(p-a_i,p-b_i)$. (However, if requiring that the
complex structure on the tangent space is $\omega$-compatible, then 
$(a_i,b_i)$ is uniquely determined.) With this understood, Theorem 4.2
says that $\{(a_i,b_i)\}$ is given by
$$
(a,b), (p-a,b-a), (p-b,p+a-b)
$$
for some $0<a<b<p$. 

With the preceding understood, Theorem 1.8 follows from the following 
proposition as we explained at the end of Section 1.

\begin{proposition}
For sufficiently large $r>0$, the $r$-version of Taubes' perturbed 
Seiberg-Witten equations associated to the square of the equivariant 
canonical bundle has a solution $((\alpha,\beta),a)$ which is fixed 
under the group action.
\end{proposition}

Assume the proposition momentarily. Letting $r\rightarrow\infty$, the
zero set $\alpha^{-1}(0)$ converges to a finite set of $J$-holomorphic 
curves $\{C_i\}$, such that $2c_1(K)=\sum_i n_i C_i$ for some integers 
$n_i>0$. Moreover, $\cup_i C_i$ is invariant under the group action
and contains all the fixed points except those $x_i$ such that
$2(a_i+b_i)=0\pmod{p}$. (Since $p$ is odd, this is equivalent to
$a_i+b_i=0\pmod{p}$.) With this understood, and with the congruence 
relation in Lemma 1.6 replaced by the following one
$$
(a_1+2)m_1+(a_2+2)m_2=0 \pmod{m},
$$
the same arguments for the proof of the Main Theorem, when applied to the 
set $\{C_i\}$ above, will yield a proof for Theorem 1.8.
(Regarding Remark 1.7, the new ``applicability'' condition which ensures
the hypothesis `$x\in\alpha^{-1}(0)$' in Lemma 1.6 is $2(m_1+m_2)\neq 0
\pmod{m}$, but again, since $p=m$ is odd, this is equivalent to the original
condition $m_1+m_2\neq 0\pmod{m}$.)

The proof of Proposition 4.3 goes as follows. Since $b_G^{2,+}=1$ in this
case, the equivariant Seiberg-Witten invariant (which is simply the 
Seiberg-Witten invariant of the orbifold $X/G$) is well-defined only
after specifying a choice of chambers. Let $E$ be an equivariant complex
line bundle over $X$. We denote by $SW^G(E)$ the equivariant Seiberg-Witten 
invariant defined using the associated $r$-version of Taubes' perturbed 
Seiberg-Witten equations with $r>0$ sufficiently large. Then the wall-crossing
formula gives
$$
|SW^G(E)\pm SW^G(K-E)|=1
$$
provided that the formal dimension $d(E)$ of the equivariant Seiberg-Witten 
moduli space is non-negative. (Thanks to Tian-Jun Li for explaining this 
to me.)

We consider the case where $E=2K$. Notice that $SW^G(K-E)$ must be zero,
because otherwise by Taubes' $SW\Rightarrow Gr$ theorem in \cite{T}, 
$c_1(K-E)=c_1(-K)$ 
is represented by $J$-holomorphic curves which contradicts the assumption
$c_1(K)\cdot [\omega]>0$. It follows that $SW^G(2K)=\pm 1$ if $d(2K)\geq 0$.

A formula for $d(E)$ may be found in Appendix A of \cite{C2} (see also
Lemma 3.3 in \cite{C1}). In the present case, we have
$$
d(2K)=\frac{1}{p}(c_1(2K)^2-c_1(2K)\cdot c_1(K)+\sum_{i=1}^3
\sum_{x=1}^{p-1}\frac{2(\mu_p^{-2(a_i+b_i)x}-1)}{(1-\mu_p^{-a_ix})
(1-\mu_p^{-b_ix})}), 
$$
where $\mu_p=\exp(2\pi i/p)$. 

Proposition 4.3 follows by showing that $d(2K)\geq 0$. In the calculation
of $d(2K)$, the fact that $\{(a_i,b_i)\}$ is given by
$$
(a,b), (p-a,b-a), (p-b,p+a-b)
$$
for some $0<a<b<p$ plays a crucial role. 

\begin{lemma}
Let $c,d$ be satisfying $0\leq c\leq p$, $0<d<p$, and let $\delta(c,d)$ be 
the unique
solution to $c-d\delta=0\pmod{p}$ for $0\leq\delta<p$. Then
$$
\sum_{x=1}^{p-1}\frac{2\mu_p^{cx}}{1-\mu_p^{-dx}}=p-1-2\delta(c,d).
$$
\end{lemma} 

\begin{proof}
Set $\phi_{c,d}(t)\equiv \sum_{x=1}^{p-1}\mu_p^{cx}(1-\mu_p^{-dx}t)^{-1}$.
Then
\begin{eqnarray*}
\phi_{c,d}(t) & = & \sum_{x=1}^{p-1}\mu_p^{cx}\sum_{l=0}^\infty
                     (\mu_p^{-dx}t)^l\\
              & = & \sum_{l=0}^\infty t^l (\sum_{x=1}^{p-1}\mu_p^{(c-dl)x})\\
              & = & \sum_{l=0}^\infty t^l (-1)+
                    \sum_{l=0}^\infty t^{\delta(c,d)+pl}\cdot p\\
              & = & \frac{1}{t-1}+\frac{pt^{\delta(c,d)}}{1-t^p}\\
              & = & \frac{t^{p-1}+\cdots+1-pt^{\delta(c,d)}}{t^p-1}.\\
\end{eqnarray*}
It follows that
\begin{eqnarray*}
\sum_{x=1}^{p-1}\frac{2\mu_p^{cx}}{1-\mu_p^{-dx}} & = & 2\phi_{c,d}(1)\\
   & = & 2 \cdot\frac{(t^{p-1}+\cdots+1-pt^{\delta(c,d)})^\prime|_{t=1}}
           {(t^p-1)^\prime|_{t=1}}\\
   & = & 2 \cdot\frac{(p-1)+\cdots + 1-p\delta(c,d)}{p}\\
   & = & p-1-2\delta(c,d).\\
\end{eqnarray*}
\end{proof}

With the preceding lemma, we compute 
\begin{eqnarray*}
\sum_{x=1}^{p-1}\frac{2(\mu_p^{-2(a_i+b_i)x}-1)}{(1-\mu_p^{-a_ix})
(1-\mu_p^{-b_ix})} & = & \sum_{x=1}^{p-1}
\frac{2(\mu_p^{(a_i+b_i)x}-1)}{(1-\mu_p^{-a_ix})
(1-\mu_p^{-b_ix})}\\
          & = & \sum_{x=1}^{p-1}\frac{2\mu_p^{(a_i+b_i)x}}{1-\mu_p^{-b_ix}}+
              \sum_{x=1}^{p-1}\frac{2\mu_p^{b_ix}}{1-\mu_p^{-a_ix}}\\
          & = & p-1-2\delta(a_i+b_i,b_i)+p-1-2\delta(b_i,a_i).
\end{eqnarray*}

Now without loss of generality, we assume
$$
(a_1,b_1)=(a,b), (a_2,b_2)=(p-a,b-a), (a_3,b_3)=(p-b,p+a-b).
$$
Then one can check directly that
$$
(\delta(a_1+b_1,b_1)+\delta(b_3,a_3))b=2b \pmod{p},
$$
which implies that 
$$
\delta(a_1+b_1,b_1)+\delta(b_3,a_3)=\left \{\begin{array}{ll}
2 & \mbox{ if } a+b=p\\
p+2 & \mbox{ if } a+b\neq p.\\
\end{array} \right .
$$
Similarly,
$$
(\delta(a_2+b_2,b_2)+\delta(a_3+b_3,b_3))(b-a)=3(b-a) \pmod{p},
$$
which implies that 
$$
\delta(a_2+b_2,b_2)+\delta(a_3+b_3,b_3)=\left \{\begin{array}{ll}
0 & \mbox{ if } b=2a \mbox{ and } p+a=2b\\
3 & \mbox{ if } b=2a \mbox{ or } p+a=2b\\
p+3 & \mbox{ if } b\neq 2a \mbox{ and } p+a\neq 2b,\\
\end{array} \right .
$$
and
$$
(\delta(b_1,a_1)+\delta(b_2,a_2))a=a \pmod{p},
$$
which implies that $\delta(b_1,a_1)+\delta(b_2,a_2)=p+1$. 

Finally, we note that $c_1(K)^2=9$, hence 
\begin{eqnarray*}
d(2K) & = & \frac{1}{p}(18+6(p-1)
-2\sum_{i=1}^3(\delta(a_i+b_i,b_i)+\delta(b_i,a_i))\\
& \geq & \frac{1}{p}(18+6(p-1)-2(3p+6))=0.\\
\end{eqnarray*}
This finishes the proof of Theorem 1.8. 

\section{Proof of Theorem 1.1}

Let $X$ be a complex surface with Kodaira dimension $\kappa(X)\geq 0$, and let $X$ be given a 
holomorphic $\Z_p$-action of prime order. Without loss of generality, we assume that $X$ is minimal
and the $\Z_p$-action is homologically trivial over $\Q$.

According to the Enriques-Kodaira classification \cite{BPV}, $X$ is either a surface of general type,
a $K3$ surface, or an elliptic surface (we used $\kappa(X)\geq 0$ here). Theorem 1.1 follows easily
if $X$ is a surface of general type or a $K3$ surface, for in the former case, Xiao's theorem provides
the bound $p\leq 42^2 c_1(K_X)^2$, and in the latter case, the $\Z_p$-action must be trivial
(cf. \cite{BPV}, Chapter VIII, Prop. 11.3). In what follows, we will focus on the remaining case where 
$X$ is elliptic. 

The basic idea goes as follows. Let $\pi:X\rightarrow \Sigma$ be an elliptic fibration, and let $g:
X\rightarrow X$ be an automorphism of $X$ which is homologically trivial over $\Q$. If $F$ is a 
fiber of $\pi$, then $g\cdot F$ must also be a fiber because $(g\cdot F)\cdot F=F\cdot F=0$. 
This shows that $g$ must be preserving the elliptic fibration $\pi:X\rightarrow \Sigma$. By analyzing
$g$ with respect to the fibration, Ueno \cite{Ueno} and Peters \cite{Peters} were able to show triviality
of $g$ in many circumstances. We take the same approach here, however, since we only need
to show triviality of $g$ when $g$ has sufficiently large order, the argument can be made much simpler. 
On the other hand, because the non-K\"{a}hler case was missing in Peters \cite{Peters}, we decided
to give an independent, self-contained proof here. 

\vspace{3mm}

{\bf Case (i)} $c_2(X)=12d>0$. Let $\bar{g}:\Sigma\rightarrow \Sigma$ be the automorphism 
induced by $g$. The first step is to show that $\bar{g}=1$.

\begin{lemma}
Let $\pi:X\rightarrow \Sigma$ be an elliptic fibration. Then $\pi^\ast: H^1(\Sigma;\Q)\rightarrow H^1(X;\Q)$
is injective. 
\end{lemma}

\begin{proof}
Let $\alpha\in  H^1(\Sigma;\Z)$ be any non-zero element, and let $\gamma$ be an embedded 
closed curve in $\Sigma$ such that $\alpha\cdot \gamma\neq 0$. (We may assume $\gamma$ 
misses all the singular values of $\pi$.) Take any lifting $\gamma^\prime$ of $\gamma$ in $X$,
i.e., $\pi(\gamma^\prime)=\gamma$. Then $\pi^\ast \alpha \cdot \gamma^\prime
=\alpha\cdot \gamma\neq 0$, which implies that $\pi^\ast \alpha\neq 0$. 
\end{proof} 

Since $g$ is homologically trivial, we deduce from Lemma 5.1 that $\bar{g}$ is also homologically trivial.
This implies that $\bar{g}=1$ when $\text{genus}(\Sigma)>1$, and that $\bar{g}$ is
a translation when $\text{genus}(\Sigma)=1$. Furthermore for the latter case, 
note that since $c_2(X)>0$, there must be a singular fiber with non-zero Euler characteristic. 
If $\bar{g}\neq 1$, then $\bar{g}$ is free of fixed points, so that each of such fibers will generate 
$p$ disjoint copies, which is impossible when $p>c_2(X)$. 

Consider the remaining case where $\text{genus}(\Sigma)=0$. Suppose $\bar{g}\neq 1$, and
let $w_1,w_2\in \Sigma$ be the two fixed points of $\bar{g}$. Let $z\in \Sigma$, $z\neq w_1,w_2$,
be any point. By the same argument as above, the fiber $\pi^{-1}(z)$ must have zero Euler
characteristic when $p>c_2(X)$. In the next lemma, we eliminate the possibility that $\pi^{-1}(z)$
is a multiple fiber with smooth reduction provided that $p> |Tor H_2(X)|$.

\begin{lemma}
Let $\pi:X\rightarrow \Sigma$ be an elliptic fibration and $g$ be an order-$p$ automorphism of $X$ preserving the fibration, where $p>5$ when $c_2(X)=0$. Suppose $z_1,\cdots,z_p\in \Sigma$ is 
a free orbit of the induced automorphism $\bar{g}:\Sigma\rightarrow\Sigma$.  If each $\pi^{-1}(z_i)$ 
is a multiple fiber, then $|Tor H_2(X)|\geq p$.
\end{lemma}

\begin{proof}
First, in order to compute the fundamental group of $X$, it is convenient to give $\Sigma$ an orbifold structure as follows. For any $z\in \Sigma$, if the fiber over $z$ is a multiple fiber of multiplicity $m$,
then $z$ is an orbifold point of multiplicity $m$. With this understood, we let $t_j$, $j=1,2,\cdots,n$,
be the set of orbifold points of $\Sigma$, and let $m_j$ be the multiplicity at $t_j$.  Note that 
each $z_i$, $i=1,2,\cdots,p$, is an orbifold point of $\Sigma$ and the multiplicity at each $z_i$ is
the same; we denote it by $m$.

With the preceding understood, we divide our discussions into two cases. 
Suppose $c_2(X)>0$. Then $\pi_\ast:\pi_1(X)\rightarrow
\pi_1^{orb}(\Sigma)$ is an isomorphism (cf. Theorem 2.3 in page 158 of \cite{FM}). Taking the 
abelianization, we see that the torsion subgroup of $H_1(X)$ is the quotient of $\oplus_{j=1}^n 
\Z_{m_j}$ by the cyclic subgroup generated by $(1,1,\cdots,1)$ (cf. Corollary 2.4 in page 158 
of \cite{FM}). By the universal coefficient theorem, $|Tor H_2(X)|\geq m(p-1)\geq p$. 

Suppose $c_2(X)=0$. Then since $p>5$, it follows that $\pi_1^{orb}(\Sigma)$ is infinite, and hence
there is an exact sequence (cf. Lemma 7.3 in page 198 of \cite{FM})
$$
\begin{array}{ccccccccc}
0 & \rightarrow & \Z\oplus\Z & \stackrel{i_\ast}{\rightarrow} & \pi_1(X) & \stackrel{\pi_\ast}{\rightarrow}
& \pi_1^{orb}(\Sigma) & \rightarrow & \{1\}.
\end{array}
$$
By a similar argument, we obtain in this case $|Tor H_2(X)|\geq m(p-3)\geq p$. 
\end{proof}

We have thus showed that assuming $p>\text{max}(c_2(X),|Tor H_2(X)|)$, 
$\pi^{-1}(w_1),\pi^{-1}(w_2)$ must be the only two singular fibers. We claim that
the monodromy around each of them must be of infinite order, i.e., it must be conjugate to
$$
\pm \left (\begin{array}{cc}
1 & b\\
0 & 1\\
\end{array} \right ), \mbox { where } b\neq 0.
$$
The argument goes as follows. If the monodromy around a fiber is of finite order, then
the Euler characteristic of the fiber is less than $12$. Thus if the monodromy around $\pi^{-1}(w_1)$
or $\pi^{-1}(w_2)$ were of finite order, $c_2(X)=12d=12$ must be true. In this case, 
$\kappa(X)\geq 0$ implies that both $\pi^{-1}(w_1),\pi^{-1}(w_2)$ must be multiple fibers. 
On the other hand, it is known that
the monodromy around a multiple fiber whose reduction is singular must be of infinite order, 
hence the claim, cf. \cite{BPV}. 

With this understood, the following lemma shows that $\bar{g}=1$ is also true 
 in the case of $\text{genus}(\Sigma)=0$.

\begin{lemma}
For any $d>0$, there exist no $b>0$, $A\in SL(2;\Z)$,  such that 
$$
\pm \left (\begin{array}{cc}
1 & -(12d-b)\\
0 & 1\\
\end{array} \right )= A^{-1}\left (\begin{array}{cc}
1 & b\\
0 & 1\\
\end{array} \right )A. 
$$
\end{lemma}

\begin{proof}
Suppose that $b, A$ exist which satisfy the above equation. 
Write $A=\left (\begin{array}{cc}
x & y\\
z & w\\
\end{array} \right )$, where $xw-yz=1$. Then $A^{-1}=\left (\begin{array}{cc}
w & -y\\
-z & x\\
\end{array} \right )$. Now 
$$
A^{-1}\left (\begin{array}{cc}
1 & b\\
0 & 1\\
\end{array} \right )A=\left (\begin{array}{cc}
w & -y\\
-z & x\\
\end{array} \right ) \left (\begin{array}{cc}
1 & b\\
0 & 1\\
\end{array} \right )\left (\begin{array}{cc}
x & y\\
z & w\\
\end{array} \right )=\left (\begin{array}{cc}
1+wzb & w^2b\\
-z^2 b & 1+wzb\\
\end{array} \right ),
$$
which gives $z=0$, $w=\pm 1$, i.e., $A=\pm I$. But this is easily seen to contradict $d>0$.

\end{proof}

Now with $\bar{g}=1$, we see that $g$ leaves each fiber of the elliptic fibration invariant. 
Moreover, assuming $p>6$, $g$ must induce a translation along each regular fiber. Finally,
since $g$ is homologically trivial over $\Q$, it follows easily that each irreducible component 
of a fiber must be invariant. 

A $\Z_p$-action on $\s^2$ has two fixed points. Simple inspection shows that a singular fiber
whose reduction is singular but not of type $I_b$, $b>0$, must contain a point $x$ such that
(1) $\pi$ is of maximal rank at $x$, (2) $x$ is a fixed point of $g$. Moreover, one can argue 
as in \cite{CK}, $\S 3$, that for large enough $p$, both weights of the action of $g$ at $x$ are
non-zero, which clearly contradicts the fact that $\bar{g}=1$. On the other hand, an invariant 
fiber whose reduction is of type $I_b$, $b>0$, either contains only isolated fixed 
points of weights $(1,p-1)$, or has an irreducible component fixed under $g$
(cf. Proposition 3.7(3) in \cite{CK}). But the latter also contradicts 
the fact that $\bar{g}=1$. 

Summing up our discussion, we see that $g$ acts on $X$ with a nonempty set of isolated fixed points,
all of weights $(1,p-1)$. As we argued in \cite{CK}, the $G$-signature theorem implies that $g=1$. 

\vspace{3mm}

{\bf Case (ii)} $c_2(X)=0$.  In this case, the only singular fibers are multiple fibers with smooth reduction.

We consider first the case where $X$ is the original complex surface, i.e., it is already minimal and
need not to be further blown down. With this understood, we divide our discussions into two cases.

(1) Suppose $X$ contains no multiple fibers. In this case, $\pi: X\rightarrow \Sigma$ is called an 
elliptic fiber bundle.
Note that since $\kappa(X)\geq 0$, we must have $\text{genus}(\Sigma)>0$. 

Let $E$ be the typical fiber of $X$ which is an elliptic curve.
Then the structure group of $X$ as an elliptic fiber bundle is 
contained in the group $A(E)$ of biholomorphisms 
of $E$. By fixing an origin $0\in E$, $E$ itself may be regarded as a normal subgroup of $A(E)$
(acting as translations on $E$), and the quotient group $A(E)/E$ is a cyclic group of order 
$n$, where 
\begin{eqnarray*}
n=4 & \mbox{ if } E=\C/(\Z\oplus \Z i)\\
n=6 & \mbox{ if } E=\C/(\Z\oplus \Z\omega), \;\;\omega=\exp(\pi i/3)\\
n=2 & \mbox{ in all other cases. }
\end{eqnarray*}
The bundle $\pi: X\rightarrow \Sigma$ is called a principal bundle if its structure group 
can be reduced to $E\subset A(E)$. 

Clearly, if $X$ is a principal bundle, it admits a holomorphic circle action 
which acts as translations on each fiber.  Assume $X$ is not a principal bundle. Then 
there exist a principal elliptic fiber bundle 
$\pi^\prime:X^\prime\rightarrow \Sigma^\prime$ and unramified coverings $\tau:X^\prime\rightarrow 
X^\prime$, $\sigma:\Sigma^\prime\rightarrow \Sigma^\prime$, such that $X=X^\prime/\tau$ and
$\Sigma=\Sigma^\prime/\sigma$. Moreover, since $X$ is not principal, the induced action of $\tau$
on $E$ defines a non-trivial element in $A(E)/E$. See \cite{BPV}, Chapter V, $\S 5$ for more details.

With this understood, suppose $\text{genus}(\Sigma)>1$. Let $g$ be an order-$p$ automorphism on
$X$ which is homologically trivial. Then by Lemma 5.1, $g$ must leave each fiber invariant. Moreover,
when $p>6$, the induced action of $g$ on the fiber $E$ must be a translation which commutes with the induced action of $\tau$ on $E$. It is clear that the order $p$ of such a translation is universally bounded.

Suppose $\text{genus}(\Sigma)=1$. Then $X$ is called bi-elliptic (see p. 119 of \cite{BPV} for a
complete list). A bi-elliptic surface admits holomorphic circle actions which induce (non-trivial) 
translations on the base $\Sigma$. This finishes the proof for the case where $X$ contains no 
multiple fibers. 

(2) Suppose $X$ contains at least one multiple fiber. First, we remark that the condition $\kappa(X)
\geq 0$ implies that, in the case of $\text{genus}(\Sigma)=0$, $X$ contains at least three multiple fibers
(cf. \cite{FM}). 
With Lemmas 5.1, 5.2, we see that there is a universal constant $c>0$ such that if $p>c(1+|Tor H_2(X)|)$,
a homologically trivial order-$p$ automorphism of $X$ must act trivially on the base $\Sigma$ and as a translation along each fiber. 
With this understood, we will reduce the proof for this case to the previous 
case using the following observation: {\it the log transforms on a regular fiber may be done equivariantly 
with respect to translations along the fiber}. 

Indeed, let $\pi^\prime: Y^\prime\rightarrow \Delta(t)$ be an elliptic fibration, where $\Delta(t)$ 
is a disc with 
coordinate $t$ centered at the origin, and $Y^\prime=\C\times \Delta(t)/(\Z\oplus \Z\omega(t))$.
Here $\omega(t)$ is a holomorphic function on $\Delta(t)$ with $Im \;\omega(t)\neq 0$. For any
$m>1$ and $k>0$ such that $m,k$ are relatively prime, a log transform with multiplicity $m$ is
done as follows. Consider 
$$
\C\times \Delta(s)/ (\Z\oplus \Z\omega(s^m))
$$
with the cyclic group action of order $m$ generated by 
$$
(z,s)\mapsto (z+k/m, \exp(2\pi i/m) s). 
$$
Let $Y$ be the quotient, which carries a natural elliptic fibration $\pi: Y
\rightarrow \Delta(t)$ given by $(z,s)\mapsto t=s^m$. Note that the fiber of $\pi$ at $t=0$ 
is a multiple fiber. With this understood, the log transform is the operation which replaces the
fiber $(\pi^\prime)^{-1}(0)$ in $Y^\prime$ by the fiber $\pi^{-1}(0)$ in $Y$ (which is a multiple fiber), 
via the fiber-preserving 
biholomorphism $f: Y\setminus \pi^{-1}(0)\rightarrow Y^\prime\setminus (\pi^\prime)^{-1}(0)$, where
$f$ is induced by $\hat{f}: \C\times (\Delta(s)\setminus \{0\})/ (\Z\oplus \Z\omega(s^m))
\rightarrow \C\times (\Delta(t)\setminus \{0\})/(\Z\oplus \Z\omega(t))$ (cf. \cite{BPV}, p. 216):
$$
\hat{f}: (z,s)\mapsto (z-(k/2\pi i)\ln s, s^m).
$$

Now suppose an automorphism $g$ of $\pi: Y\rightarrow \Delta(t)$ is given by 
a translation along the fiber. Then on the ramified cover $\C\times \Delta(s)/ (\Z\oplus \Z\omega(s^m))$,
it can be written as 
$$
g: (z,s)\mapsto (z+u(s^m), s)
$$ 
for some holomorphic function $u(t)$ over $\Delta(t)$. Let $g^\prime:(z,t)\mapsto (z+u(t),t)$ be the
automorphism of $\C\times (\Delta(t)\setminus \{0\})/(\Z\oplus \Z\omega(t))$. Then one has 
$g^\prime\circ \hat{f}=\hat{f}\circ g$. Hence after an inverse log transform performed to 
$\pi: Y\rightarrow\Delta(t)$, there is an induced automorphism $g^\prime$ of 
$\pi^\prime: Y^\prime\rightarrow \Delta(t)$, which is given by the translation 
$$
(z,t)\mapsto (z+u(t), t). 
$$ 

With the preceding understood, suppose $g$ is an order-$p$ automorphism of $X$ which acts as
a translation on each fiber. We perform an inverse log transform to $X$ at each multiple fiber, and let
$\pi^\prime:X^\prime\rightarrow \Sigma$ be the resulting elliptic surface. Then $\pi^\prime:X^\prime\rightarrow \Sigma$ is an elliptic fiber bundle, and by our previous discussion, $X^\prime$ inherits
an order-$p$ automorphism $g^\prime$ which acts as a translation on each fiber. Now observe 
that there are two possibilities: (i) $X^\prime$ is a principle bundle and there is a holomorphic circle 
action acting as translations along each fiber, or (ii) there is a universal constant $c>0$ such that 
the order of $g^\prime$ is bounded by $c$, i.e., $p\leq c$.
In the former case, it follows that $X$ admits a holomorphic circle action, and in the latter case,
we see that the order of $g$ is universally bounded. This finishes the proof for the case where
$X$ contains at least one multiple fiber. 

Finally, we consider the case where the original complex surface is not minimal. 
In this case, the key observation is that after blowing down, the induced $\Z_p$-action always 
has a fixed point which is isolated. (The fixed point is the image of an exceptional divisor under the
blowing down.)  On the other hand, from our discussions above, we in fact proved more: {\it 
if the order of a homologically trivial automorphism $g$ of a minimal elliptic surface of $c_2=0$ and
$\kappa\geq 0$ is greater than $c(1+|Tor\; H_2|)$ for some universal constant $c>0$, then $g$ must 
come from a fixed-point free holomorphic circle action on the surface. Moreover, the set of exceptional 
orbits of the circle action is either empty or a union of embedded tori}.

This finishes the proof of  Theorem 1.1.

\vspace{3mm}

\noindent{\bf Acknowledgments}: The author has benefited tremendously from 
collaborations with Slawomir Kwasik \cite{CK,CK1,CK2}, whom he wishes 
to thank warmly. The author is also very grateful to Tian-Jun Li for 
his interests in this work and for several helpful communications, 
and particularly, for pointing out an error in the original
version of this paper. Part of this work was carried out 
during the author's visit to the Max Planck Institute for Mathematics 
in the Sciences, Leipzig. The author wishes to thank the institute for 
its hospitality and financial support during the visit. Finally, special thanks
go to an anonymous referee for his/her kind communication regarding the repeated 
knot surgery discussed in Example 1.3.

\vspace{2mm}

{\Small University of Massachusetts, Amherst.\\
{\it E-mail:} wchen@math.umass.edu

\end{document}